\documentclass[review,hidelinks,onefignum,onetabnum]{siamart250211}

\usepackage{float}
\usepackage{subcaption}
\usepackage{multirow} 
\usepackage{array}      
\usepackage{booktabs}
\usepackage[skip=2pt]{caption}
\usepackage{xr-hyper} 
\usepackage{hyperref}
\usepackage{caption}
\usepackage[utf8]{inputenc}
\usepackage{bm}

\usepackage{lipsum}
\usepackage{amsfonts}
\usepackage{graphicx}
\usepackage{epstopdf}
\usepackage{algorithmic}
\ifpdf
  \DeclareGraphicsExtensions{.eps,.pdf,.png,.jpg}
\else
  \DeclareGraphicsExtensions{.eps}
\fi

\newsiamremark{remark}{Remark}
\newsiamremark{hypothesis}{Hypothesis}
\crefname{hypothesis}{Hypothesis}{Hypotheses}
\newsiamthm{claim}{Claim}
\newsiamremark{fact}{Fact}
\crefname{fact}{Fact}{Facts}

\headers{Parametric Matrix Equations}{Xuelian Wen, Qiuqi Li, Juan Zhang}

\title{Data-driven Model Reduction for Parameter-Dependent Matrix Equations via Operator Inference\thanks{Submitted to the editors DATE.
\funding{This work was supported in part by the National Natural Science Foundation of China (12471405, 2023YFB3001604)
and the Science and Technology Innovation Program of Hunan Province (2025RC3080).}}}

\author{Xuelian Wen\thanks{Key Laboratory of Intelligent Computing and Information Processing of Ministry of Education, Key Laboratory for Computation and Simulation in Science and Engineering, the School of Mathematics and Computational Science, Xiangtan University, Xiangtan, Hunan, China
  (\email{202421511236@smail.xtu.edu.cn}, \email{zhangjuan@xtu.edu.cn})}.
\and Qiuqi Li\thanks{the School of Mathematics, Hunan University, Changsha, Hunan, China
  (\email{qli28@hnu.edu.cn})}.
\and Juan Zhang\footnotemark[2]}

\usepackage{amsopn}

\ifpdf
\hypersetup{
  pdftitle={Data-driven Model Reduction for Parameter-Dependent Matrix Equations via Operator Inference},
  pdfauthor={Xuelian Wen, Qiuqi Li, Juan Zhang}
}
\fi

\usepackage{xr}

\begin{document}
\maketitle

\begin{abstract}
This work develops a non-intrusive, data-driven surrogate modeling framework based on Operator Inference (OpInf) for rapidly solving parameter-dependent matrix equations in many-query settings. Motivated by the requirements of the OpInf methodology, we reformulate the matrix equations into a structured representation that explicitly shows the parameter dependence in polynomial form. This reformulation is crucial for efficient model reduction. This approach constructs reduced-order models via regression on solution snapshots, bypassing the need for expensive full-order operators and thus overcoming the primary bottlenecks of intrusive methods in high-dimensional contexts. Numerical experiments confirm their accuracy and computational efficiency, demonstrating that our work is a scalable and practical solution for parameter-dependent matrix equations.
\end{abstract}



\section{Introduction}
Matrix equations constitute a fundamental component of systems and control theory, among which their parameter-dependent variants play an increasingly critical role in the analysis and design of complex systems subject to uncertainties or varying operating conditions. Prominent examples include the parameter-dependent algebraic Lyapunov and Riccati equations (PALEs and PAREs). PALEs are indispensable in various applications, such as the design of low-gain feedback controllers \cite{Zhou2008}, vibrational analysis \cite{Truhar2004}, parametric model reduction \cite{Son2021}, and studies of multi-agent systems \cite{Kim2011,Su2013}. Similarly, PAREs are central to problems in optimal control, state estimation, and robust control design \cite{Anderson2007,Khargonekar1990}. This paper focuses on the following class of parameter-dependent matrix equations:
\begin{equation}
\mathcal{L}_{i}(X_1,\cdots,X_s;\mu) + Q_i(\mu) = 0_{n\times n}, \quad i \in \mathcal{S} = \{1,2,\cdots,s\}, 
\label{1}
\end{equation}
where the parameter $\mu = [\mu_1, \dots, \mu_d]^\top$ belongs to a compact domain $\mathcal{P} \subset \mathbb{R}^d$, for each $i\in\mathcal{S}:$
\begin{itemize}
    \item $X_i(\mu)\in\mathbb{R}^{n\times n}$ denotes the unknown matrix to be solved;
    \item  
    $Q_i(\mu)=M_i^\top(\mu)M_i(\mu)$, with $M_i(\mu)=\sum_{j=1}^{n_{M_i}}\theta_j^{M_i}(\mu)M_{i,j}
    \in\mathbb{R}^{l\times n}$, admits an affine parameterization;
    \item $\mathcal{L}_{i}: (\mathbb{R}^{n \times n})^s \times \mathcal{P}\to \mathbb{R}^{n \times n}$ is a parameter-dependent matrix operator. 
\end{itemize}

\Cref{1} provides a general framework that encapsulates several important types of PALEs and PAREs. Specifically, the continuous-time PALEs:
$$
A^T(\mu)X(\mu)+X(\mu)A(\mu)+Q(\mu)=0_{n\times n},
$$
the discrete-time PALEs:
$$
A^T(\mu)X(\mu)A(\mu)-X(\mu)+Q(\mu)=0_{n\times n},
$$
the continuous-time coupled PALEs:
$$
A_i^\top(\mu)X_i(\mu)+X_i(\mu)A_i(\mu)+\sum_{j=1}^s\pi_{ij}X_j(\mu)+Q_i(\mu)=0_{n\times n},
$$
and the continuous-time PAREs:
$$
A^T(\mu)X(\mu)+X(\mu)A(\mu)-X(\mu)G(\mu)X(\mu)+Q(\mu)=0_{n\times n},
$$
in which the matrices $A(\mu),\ A_i(\mu),\ G(\mu)$ all conform to the affine parameterization, similar to that of $M_i(\mu)$. For the four classes of equations under consideration, we assume that standard well-posedness conditions are satisfied for all $\mu \in \mathcal{P}$. Under these assumptions, the three types of PALEs are guaranteed to admit unique solutions. In contrast, the continuous-time PAREs generally allow for multiple solutions—including indefinite ones—so we further impose the conventional requirement that ensures the existence of a unique symmetric positive semidefinite solution.

Given the central role of algebraic Lyapunov and Riccati equations, substantial effort has been devoted to developing numerical algorithms for their solutions. Numerical methods for the algebraic Lyapunov equations and the algebraic Riccati equations are well-established. The direct methods for Lyapunov equations, such as the Bartels-Stewart algorithm \cite{Bartels1972} and the Hammarling algorithm \cite{Hammarling1982}, are effective. The iterative techniques, including Krylov subspace methods \cite{Kasenally1994} and the alternating direction implicit method \cite{Lu1991,Wachspress1988}, are essential. Riccati equations are similarly addressed through both exact methods, such as Laub’s Schur decomposition \cite{Laub1979} and Benner’s generalized exact line search \cite{Benner2002}, as well as large-scale iterative methods like rational Krylov subspace projection \cite{Simoncini2016} and the matrix sign function method \cite{Byers1987}. These approaches, comprehensively surveyed in \cite{Benner2008}, form a solid computational foundation.

However, solving parametric systems introduces significant computational challenges. Traditional methods necessitate resolving the full-order equations for each parameter value, leading to prohibitive computational costs for many-query tasks such as parameter sweeps, design optimization, and uncertainty quantification. This limitation has motivated the integration of model reduction techniques with matrix equation solvers to enable efficient parametric modeling.

Model Order Reduction (MOR) aims to construct low-dimensional surrogate models that approximate the input–output behavior of high-dimensional systems. MOR techniques are broadly classified into intrusive and non-intrusive approaches. Intrusive methods require access to and manipulation of the full-order system operators. Representative methods include balanced truncation based on Lyapunov equations \cite{Gugercin2004,Moore2003}, algorithms tailored to large-scale systems \cite{Penzl2006}, extensions to bilinear and stochastic systems \cite{Benner2011}, interpolatory frameworks for parametric systems \cite{Baur2011}, as well as techniques for second-order and descriptor systems \cite{Benner2011b,Heinkenschloss2008}. Galerkin projection techniques, such as POD-Galerkin \cite{Kunisch2001}, employ Proper Orthogonal Decomposition (POD) \cite{Berkooz1993,Lumley1967,Rathinam2003,Sirovich1987} to construct optimal reduced bases from solution snapshots. The Reduced Basis (RB) method \cite{Buffa2012,Grepl2007} further enables efficient approximation for parametric problems by projecting high-dimensional systems onto problem-specific, low-dimensional subspaces. Nevertheless, intrusive methods generally require repeated evaluations of the high-fidelity model, which remain computationally demanding. To alleviate this cost, non-intrusive MOR methods have been developed. These methods  rely solely on parameter–solution pairs without explicit knowledge of system operators. Representative techniques include the Loewner framework \cite{Ionita2014} and Dynamic Mode Decomposition \cite{Schmid2010}. Among them, Operator Inference (OpInf) is a leading approach that learns reduced-order operators directly from trajectory data via regression \cite{Benner2020,McQuarrie2021,Mcquarrie2023,Peherstorfer2016,Swischuk2020}. A key strength of OpInf lies in its ability to preserve physical structure. These attributes render OpInf particularly well-suited for parametric systems where intrusive model reduction is infeasible.

Various intrusive reduction strategies have been proposed for parametric matrix equations. For instance, Kressner et al. utilized tensor compression for multi-parameter Lyapunov equations \cite{Kressner2014}; Nguyen et al. developed RB methods with error estimators for affine systems \cite{Son2017}; and Schmidt et al. applied related ideas to parametric Riccati equations \cite{Schmidt2018}. More recently, Palitta et al. introduced recycling Krylov techniques for sequences of parameterized Lyapunov equations \cite{Palitta2025}. While valuable, these intrusive frameworks still encounter challenges concerning computational overhead and the storage of high-dimensional operators.

The primary contributions of this work are threefold: firstly, we develop an efficient computational framework based on OpInf for solving parameter-dependent matrix equations that demonstrates effectiveness across both one-dimensional and two-dimensional parameter spaces. The approach offers particular advantages in many-query scenarios where rapid solution evaluation is required for multiple parameter values. Secondly, we enable the direct application of OpInf by reformulating the original parameter-dependent matrix equations to explicitly expose their underlying polynomial dependence. This reformulation is achieved through a vectorization step that reveals the required polynomial structure without necessitating computations in the resulting high-dimensional space. Finally, OpInf avoids the requirements of large-scale matrix operations and full-order operator storage, which are inherent to traditional intrusive reduction methods and become infeasible at high dimensions due to memory constraints. OpInf learns reduced operators solely from solution snapshots and the affine parameterization functions, drastically reducing computational and storage costs.

The paper is structured as follows: in \cref{sec2}, we introduce the relevant methodological framework. The procedure of the OpInf method for affine-parameterized matrix equations is  detailed in \cref{sec3}. In \cref{sec4}, we present four numerical examples to validate the effectiveness of applying OpInf method to matrix equations, in terms of both CPU time and computational accuracy. In \cref{sec5}, we summarize the core methodology and outline future research directions.

\section{The Operator Inference Framework}
\label{sec2} 
High-fidelity numerical methods (e.g., Bartels-Stewart method, Schur method, and Hamiltonian Schur method) for solving parameter-dependent  matrix equations $\cref{1}$ often involve operations with large-scale coefficient matrices and require considerable computational resources, particularly in many-query contexts such as uncertainty quantification, parameter estimation, or robust control. 
These tasks necessitate repeated solutions over the parameter space $\mathcal{P}$. Therefore, we focus on developing a framework for rapidly solving parameter-dependent matrix equations $\cref{1}$.

OpInf is a non-intrusive, data-driven model reduction technique that learns low-dimensional operators directly from snapshot data to form a reduced-order model that preserves the structure and dynamics of the original system. This approach is particularly well-suited for systems exhibiting polynomial nonlinearities, as it allows the structure of the full-order model to be preserved exactly in the reduced-order setting. The key idea is to postulate a low-dimensional model with the same polynomial form as the full-order system and infer its operators via a regression problem that minimizes the residual of the reduced-order equations over the available data. To enable the application of OpInf, $\cref{1}$ must first be reformulated into a representation that makes its polynomial structure explicit:
\begin{equation}
    C_2(\mu)x^2(\mu)+C_{1}(\mu)x(\mu) + C_{0}(\mu) = x(\mu),
    \label{polynomial structure}
\end{equation}
where 
$$
x(\mu)=\begin{bmatrix}
    \operatorname{vec}^\top(X_1\left(\mu\right)) & \cdots &\operatorname{vec}^\top(X_s\left(\mu\right))
\end{bmatrix}^\top.
$$ 
Thus, the dimension of the full-order system \cref{polynomial structure} is $N = s n^2$. Our goal is to construct an efficient surrogate model that yields $\widetilde{x}(\mu)$ approximations of $x(\mu)$ without solving the full-order model at each new parameter.
 
The reduced-order model (ROM) is constructed in a low-dimensional subspace spanned by basis vectors $\{v_j\}_{j=1}^r$, obtainable via methods such as POD, greedy algorithms, or randomized sampling with orthogonalization. The surrogate solution is then represented as
$$
\widetilde{x}(\mu)=\sum_{j=1}^r\hat{x}_{j}(\mu)v_{j},
$$    
where the coefficients  $\{\hat{x}_{j}(\mu)\}_{j=1}^r$ are determined by solving the learned reduced-order operator equation:
\begin{equation}
    \widehat{C}_2(\mu)\hat{x}^2(\mu)+\widehat{C}_1\hat{x}(\mu)+\widehat{C}_0(\mu)=\hat{x}(\mu),
    \label{reduce}
\end{equation}
which maintains the same polynomial form as the full-order system. 

The operators in $\cref{reduce}$ are inferred by solving a  regression problem that ensures consistency with the projected snapshot data. This results in a computationally efficient and accurate surrogate model that can be rapidly evaluated for new parameter values.

\section{OpInf for parameter-dependent matrix equations}
\label{sec3}
The application of the OpInf framework necessitates a reformulation of system $\cref{1}$ to explicitly reveal its polynomial structure. This is effectively achieved through vectorization, which clearly exposes the inherent polynomial relationships.
\Cref{sec3.1} details the vectorization of the parameter-dependent matrix equations into a polynomial system, establishing the mathematical foundation for operator learning. \Cref{sec3.2} then introduces reduced-order basis construction methods and describes the intrusive projection approach for building the reduced-order model. Finally, \cref{sec3.3} provides a comprehensive description of the non-intrusive OpInf method for learning reduced operators from data and establishes relevant theoretical guarantees.
\subsection{Vectorization of parameter-dependence matrix equations}
\label{sec3.1}

We begin by applying the vectorization operator to each equation in $\cref{1}$. The affine parameter dependence of the matrices composing 
$\mathcal{L}_i$ and $Q_i$ allows the linear terms to be readily expressed in vectorized form using the standard identity $\operatorname{vec}(AXB) = (B^\top \otimes A)\operatorname{vec}(X)$. Having addressed the linear part, the central challenge in this vectorization process lies in the treatment of the quadratic terms, such as the $XGX$ term appearing in algebraic Riccati equations. The following lemma outlines the vectorization procedure under the assumption that $X$ is symmetric.
\begin{lemma}[Vectorization of quadratic term ]\label{XGX}
    Let $X\in\mathbb{R}^{n\times n}$ be a symmetric matrix, and let $e_i$ denote the $i$-th column of the identity matrix $I_n$. Then, the vectorization of the quadratic term $XGX$ satisfies
    $$
    \operatorname{vec}(XGX)=\sum\limits_{i=1}^n\sum\limits_{j=1}^ng_{ij}E_{ij}x\otimes x,
    $$
\noindent where
$
G=[g_{ij}]\in\mathbb{R}^{n\times n},\ x=\operatorname{vec}(X),$ and $ E_{ij}=(\boldsymbol{e}_j^\top\otimes I_n)\otimes (\boldsymbol{e}_i^\top\otimes I_n)\in\mathbb{R}^{n^2\times n^4}.$ 
Furthermore, owing to the redundancy in $x\otimes x$, this expression can be condensed into the following simplified quadratic operator:
$$
\operatorname{vec}(XGX)=Hx^2,
$$
where
$$H\in\mathbb{R}^{n^2\times m},\ x^2= x\hat{\otimes} x = 
\begin{bmatrix}
(x^{(1)})^\top &
\cdots &
(x^{(n^2)})^\top
\end{bmatrix}^\top\in\mathbb{R}^m,
$$
with
$$
m=\frac{n^2(n^2+1)}{2},\
x^{(i)} = x_i 
\begin{bmatrix}
x_1 &
\cdots &
x_i
\end{bmatrix}^\top \in \mathbb{R}^i,\ i=1,\ldots,n^2.
$$
\end{lemma}
Here, the symbol $\hat{\otimes}$ denotes the symmetric Kronecker product, which eliminates duplicate entries from $x\otimes x.$ Specifically, the vector $x^{(i)}$ is the product of the scalar $x_i$ with the vector comprising the first $i$ elements of $x$.
\begin{remark}
    The \cref{XGX} is derived based on the results presented in \cite{Brewer2003,Peherstorfer2016}.
\end{remark}

By leveraging \cref{XGX} and the standard vectorization identity for Kronecker products, the vectorized form of $\cref{1}$ can be expressed as
\begin{equation}
    P_{i}(\mu)x_i^2(\mu) + \sum_{j=1}^sR_{i,j}(\mu)x_j(\mu) + N_i(\mu) = 0_{n^2},\quad i=1,\cdots,s,
    \label{2.1}
\end{equation}
where $x_i(\mu)=\operatorname{vec}(X_i(\mu))$, and $x_i^2(\mu)$ is similar to $x^2$ in \cref{XGX}. In this formulation, for the $i$-th equation:
\begin{itemize}
    \item $P_i(\mu)\in\mathbb{R}^{n^2\times m}$ denotes the coefficients of the quadratic terms;
    \item $R_{i,j}(\mu)\in\mathbb{R}^{n^2\times n^2}$ denotes the coefficients linking the linear term $x_j(\mu)$ to the $i$-th equation;
    \item $N_i(\mu) = \operatorname{vec}(Q_i(\mu))\in\mathbb{R}^{n^2}$ denotes a parameterized term dependent on $\mu$.
\end{itemize}

Therefore, the vectorized system $\cref{2.1}$ can be assembled into the following block-structured equation:
\begin{equation}
    C_2(\mu)x^2(\mu)+\overline{C}_{1}(\mu)x(\mu) + C_{0}(\mu) = 0_{N},
    \label{3.2}
\end{equation}
where the block matrices and vectors are defined as
$$
C_2(\mu)=
\begin{bmatrix}
    P_{1}(\mu)\\
    & \ddots\\
    & & P_{s}(\mu)
\end{bmatrix},\ 
\overline{C}_1(\mu)=
\begin{bmatrix}
    R_{1,1}(\mu) & \cdots & R_{1,s}(\mu)\\
    \vdots & \ddots & \vdots\\
    R_{s,1}(\mu) & \cdots & R_{s,s}(\mu)
\end{bmatrix},
$$
$$
C_0(\mu)=
\begin{bmatrix}
    N_1(\mu)\\
    \vdots\\
    N_s(\mu)
\end{bmatrix},\
x^2(\mu)=
\begin{bmatrix}
    x_1^2(\mu)\\
    \vdots\\
    x_s^2(\mu)
\end{bmatrix},\ 
x(\mu)=
\begin{bmatrix}
    x_1(\mu)\\
    \vdots\\
    x_s(\mu)
\end{bmatrix}.
$$

Unlike parameter-dependent partial differential equations, for which the time derivative provides a natural target for OpInf, the algebraic systems considered here lack an explicitly evolving quantity. Consequently, we designate the state vector $x(\mu)$ itself as the inference target, reformulating the problem to fit data-driven learning frameworks. This is achieved by rewriting $\cref{3.2}$ into the polynomial structure $\cref{polynomial structure}$. Note that in this reformulation, the operator $C_1(\mu)$ is given by $C_1(\mu)=\overline{C}_1(\mu)+I_{N}.$ The affine parameter dependence of the original matrix operators consequently implies that the operators in \cref{polynomial structure} also admit an affine decomposition:
$$
C_2(\mu)=\sum\limits_{i=1}^{n_{C_2}}\theta_i^{C_2}(\mu) C_{2,i},\ 
C_1(\mu)=\sum\limits_{i=1}^{n_{C_1}}\theta_i^{C_1}(\mu) C_{1,i},\
C_0(\mu)=\sum\limits_{i=1}^{n_{C_0}}\theta_i^{C_0}(\mu) C_{0,i}.    $$
\subsection{Reduced-order basis construction}
\label{sec3.2}
This section outlines the construction of a reduced-order basis. Several methods exist for this purpose. POD processes all available solution snapshots via singular value decomposition to derive a globally optimal orthogonal basis, ensuring the most compact and accurate representation of the training data. A greedy selection strategy can be applied directly to the existing snapshot set to iteratively choose a representative subset, which can be computationally efficient for very high-dimensional problems. Similarly, randomized sampling offers another efficient pathway for basis generation. However, both of these approaches yield suboptimal bases compared to the provable optimality of POD. Therefore, POD remains the preferred choice due to its robustness and superior accuracy, providing the most reliable foundation for the subsequent operator regression.

We now describe this POD construction in detail. Consider the known solution vectors $\{x(\mu_i)\}_{i=1}^k$ of the system $\cref{polynomial structure}$ corresponding to parameters \(\{\mu_i\}_{i=1}^k\). These are assembled into the snapshot matrix 
$X=\begin{bmatrix}
    x(\mu_1) & \cdots & x(\mu_k)
\end{bmatrix}\in \mathbb{R}^{N \times k}$, whose singular value decomposition (SVD) is given by
$
X = \sum_{i=1}^k \sigma_i u_i w_i^\top,
$
with singular values \(\sigma_1 \geq \sigma_2 \geq \cdots \geq \sigma_k\geq 0\). The optimal rank-$r$ POD basis \(V_r\in\mathbb{R}^{N\times r}\) is obtained by solving 
$$
V_r=\mathop{\arg\min}_{\substack{\widetilde{V}^\top\widetilde{V}=I_r\\}}\Vert X-\widetilde{V}\widetilde{V}^\top X\Vert_F^2.
$$

According to the Eckart-Young-Mirsky theorem \cite{eckart1936, golub2013}, \(V_r\) is formed by the first \(r\) left singular vectors \(\{u_i\}_{i=1}^r\), and the associated reconstruction error is
\begin{equation}
\Vert X - V_r V_r^\top X \Vert_F^2=\sum_{i=r+1}^k \sigma_i^2,
\label{2.4}
\end{equation}
which shows that the accuracy of the POD basis is governed by the decay of the singular values. The following theorem provides an energy-based criterion for selecting the mode count $r$.
\begin{theorem}[Energy-based POD truncation] \label{thm:pod}
Let \(X \in \mathbb{R}^{N \times k}\) be a snapshot matrix with the singular value decomposition \(X = \sum_{i=1}^k\sigma_iu_iw_i^\top\), ane let \(\epsilon \in [0,1)\)  be a energy tolerance.  If $r$ is the smallest integer such that
\[
\frac{\sum_{i=1}^r \sigma_i^2}{\sum_{i=1}^k \sigma_i^2} 
\geq 1 - \epsilon,
\]
then the relative error of the rank-$r$ POD basis $V_r$ satisfies $$\frac{\Vert X - V_r V_r^\top X \Vert_F^2}{\Vert X \Vert_F^2}\leq \epsilon.$$
\end{theorem}

\begin{proof}
The result follows directly from \cref{2.4}.
\end{proof}
\begin{remark}
    While the POD provides a globally optimal basis, the computational cost of SVD becomes prohibitive for very high-dimensional systems. When the target accuracy permits, greedy selection from snapshot data is efficient and alternative.
\end{remark}

Given a basis matrix $V_r=[v]_{ij}$, an
intrusive POD-Galerkin projection of \cref{polynomial structure} yields a reduced system as described in \cref{reduce} with the following operators:
\begin{gather}
    \widehat{C}_0(\mu)=V_r^\top C_0(\mu)=\sum\limits_{i=1}^{n_{C_0}}\theta_i^{C_0}(\mu) \widehat{C}_{0,i}\in\mathbb{R}^r,
    \label{C0}\\
    \widehat{C}_1(\mu)=V_r^\top C_1(\mu)V_r=\sum\limits_{i=1}^{n_{C_1}}\theta_i^{C_1}(\mu) \widehat{C}_{1,i}\in\mathbb{R}^{r\times r},\label{C1}\\
    \widehat{C}_2(\mu)=\sum\limits_{i=1}^{n_{C_2}}\theta_i^{C_2}(\mu) \widehat{C}_{2,i}\in\mathbb{R}^{r\times q(r)},\label{C2}
\end{gather}
where $ q(r) = \frac{r(r+1)}{2} $, and the specific construction methods of $\widehat{C}_2(\mu)$ follow \cite{Mcquarrie2023,Peherstorfer2016}.

However, such intrusive projection requires explicit access to the  high-dimensional operators, which is often computationally expensive or even infeasible. To overcome this limitation, we employ OpInf to learn the reduced operators $\widehat{C}_{0,i}$, $\widehat{C}_{1,i}$, $\widehat{C}_{2,i}$ directly from the projected data ${\hat{x}(\mu_i)}$ and the parameter functions $\theta_i^{C_0}(\mu),\ \theta_i^{C_1}(\mu),\ \theta_i^{C_2}(\mu)$, bypassing the expensive operator construction. This non-intrusive approach is detailed in the next subsection.
\subsection{Non-intrusive Operator Inference}
\label{sec3.3}
This section details the OpInf for learning a ROM from the projected data and the parameter functions. In contrast to the intrusive projection method, OpInf determines the reduced operators by solving a data-driven regression problem, avoiding explicit dependence on the high-dimensional operators.

Consider the reduced-order system $\cref{reduce}$ with known affine parameter functions $\theta=\{\theta_j^{C_2},\theta_j^{C_2},\theta_j^{C_0}\}$. We define $\cref{reduce}$ as
\begin{equation}
    \hat{x}(\mu)=F(\widehat{O};\hat{x},\theta,\mu )=\widehat{C}_2(\mu)\hat{x}^2(\mu)+\widehat{C}_1(\mu)\hat{x}(\mu)+\widehat{C}_0(\mu),
    \label{F}
\end{equation}
where the unknown operators are consolidated into the matrix:
\begin{equation}
    \widehat{O}=
	\begin{bmatrix}
		\widehat{C}_{2,1} & 
		\cdots & 
		\widehat{C}_{2,n_{C_2}} & 
		\widehat{C}_{1,1} & 
		\cdots & 
		\widehat{C}_{1,n_{C_1}} &
        \widehat{C}_{0,1} & 
		\cdots & 
		\widehat{C}_{0,n_{C_0}}
	\end{bmatrix}^\top\in\mathbb{R}^{p(r)\times r},
    \label{O}
\end{equation}
with $p(r)=q(r)\cdot n_{C_2}+r\cdot n_{C_1}+n_{C_0}.$ 

For each parameter sample $\mu_j\in\{\mu_i\}_{i=1}^k$, with full-order solution $x_j=x(\mu_j)$, the corresponding reduced-state vectors are given by
\[
\hat{x}_j = V_r^Tx_j,\quad \hat{x}_j^2=\hat{x}_j\hat{\otimes}\hat{x}_j, \quad j=1, \cdots, k,
\]
which are assembled into the data matrices:
\[
\widehat{X} = 
\begin{bmatrix}
    \hat{x}_1 &
    \cdots & 
    \hat{x}_k
\end{bmatrix}^\top \in\mathbb{R}^{k\times r},\
\widehat{X}^2 = 
\begin{bmatrix}
    \hat{x}^2_1 &
    \cdots &
    \hat{x}^2_k
\end{bmatrix}^\top\in \mathbb{R}^{k \times q(r)}.
\]
The parameter vectors for each sample are defined as
$$
\theta^{C_2}(\mu_i)=
	\begin{bmatrix}
		\theta_1^{C_2}(\mu_i) & \cdots & \theta_{n_{C_2}}^{C_2}(\mu_i)
	\end{bmatrix}\in\mathbb{R}^{1\times n_{C_2}},
$$
$$
\theta^{C_1}(\mu_i)=
	\begin{bmatrix}
		\theta_1^{C_1}(\mu_i) & \cdots & \theta_{n_{C_1}}^{C_1}(\mu_i)
	\end{bmatrix}\in\mathbb{R}^{1\times n_{C_1}},
$$
$$
\theta^{C_0}(\mu_i)=
	\begin{bmatrix}
		\theta_1^{C_0}(\mu_i) & \cdots & \theta_{n_{C_0}}^{C_0}(\mu_i)
	\end{bmatrix}\in\mathbb{R}^{1\times n_{C_0}}.
$$

The reduced operator $\widehat{O}$ defined in $\cref{O}$ is obtained by solving the following least-squares problem:
\begin{equation}
\min_{\hat{O} \in \mathbb{R}^{p(r) \times r}} \| D \widehat{O} - \widehat{X} \|_{F}^{2},
\label{2.6}
\end{equation}
where the matrix $D$ is constructed as
$$
	D=
	\begin{bmatrix}
		D_{C_2}\mid D_{C_1} \mid D_{C_0}
	\end{bmatrix}
	=
	\begin{bmatrix}
		\theta^{C_2}(\mu_1)\otimes (\hat{x}_1^2)^\top & \theta^{C_1}(\mu_1)\otimes (\hat{x}_1)^\top & \theta^{C_0}(\mu_1)\\
		\vdots & \vdots & \vdots\\
		\theta^{C_2}(\mu_k)\otimes (\hat{x}_k^2)^\top & \theta^{C_1}(\mu_k)\otimes (\hat{x}_k)^\top & \theta^{C_0}(\mu_k)
	\end{bmatrix}.
$$

If the matrix D has full column rank, then $\eqref{2.6}$ has a unique solution. The following theorem establishes conditions for this property.
\begin{theorem} \label{theorem3.5}
Define the parameter matrices \(\Theta^{C_2} \in \mathbb{R}^{k \times n_{C_2}}\), \(\Theta^{C_1} \in \mathbb{R}^{k \times n_{C_1}}\), \(\Theta^{C_0} \in \mathbb{R}^{k \times n_{C_0}}\) with entries $[\Theta^{C_2}]_{ij} = \theta_j^{C_2}(\mu_i)$, and similarly for $\Theta^{C_1},\ \Theta^{C_0}.$ Then:
\begin{enumerate}
    \item (Necessary conditions) If $D$ has full column rank, then each of the matrices $\Theta^{C_{2}}$, $\Theta^{C_{1}}$, $\Theta^{C_{0}}$, $\widehat{X}^{2}$, and $\widehat{X}$ must have full column rank.

    \item (Sufficient condition) If for every nonzero $M_{2} \in \mathbb{R}^{q(r) \times n_{C_{2}}}$, $M_{1} \in \mathbb{R}^{r \times n_{C_{1}}}$, and $\boldsymbol{v}_{C_{0}} \in \mathbb{R}^{n_{C_{0}}}$, there exists at least one sample $\mu_{i}$ such that
    \begin{equation}
        (\hat{x}_i^2)^{\top}M_2\theta^{C_2}(\mu_i)^{\top}+(\hat{x}_i)^{\top}M_1\theta^{C_1}(\mu_i)^{\top}+\theta^{C_0}(\mu_i)v_{C_0}\neq0,
        \label{Sufficient condition}
    \end{equation}
    then $D$ has full column rank.
\end{enumerate}
\end{theorem}
\begin{proof}
    The necessary conditions follow from ~\cite{Mcquarrie2023}. To prove sufficiency, we proceed by proving its contrapositive. Assume $D$ does not have full column rank. Then there exists a nonzero vector $v = \left[v_{C_2}^\top, v_{C_1}^\top, v_{C_0}^\top\right]^\top \in \mathbb{R}^{p(r)}$ such that $Dv = 0_k$. Reshape $v_{C_2}$ into $M_2 \in \mathbb{R}^{q(r) \times n_{C_2}}$ and $v_{C_1}$ into $M_1 \in \mathbb{R}^{r \times n_{C_1}}$. Using Kronecker product properties, this implies for all samples $\mu_i$:
    $$
    (\hat{x}_i^2)^\top M_2 \theta^{C_2}(\mu_i)^\top+(\hat{x}_i)^\top M_1 \theta^{C_1}(\mu_i)^\top+\theta^{C_0}(\mu_i)v_{C_0} = 0,
    $$
    violating the sufficient condition. Therefore, $D$ has full column rank.
\end{proof}

ROM learned through OpInf and those constructed via intrusive projection, i.e., by explicitly evaluating $\cref{reduce}$, are related in the following theorem, which establishes conditions for the existence and uniqueness of the reduced-order operators.
\begin{theorem}[Existence and Uniqueness of the Reduced-Order Operators]\label{Existence and Uniqueness}
Consider the affine-parametrized matrix equation $\cref{1}$ and its vector form $\cref{polynomial structure}$. Suppose there exists a reduced-order subspace $\mathcal{M} \subset \mathbb{R}^{sn^2}$ spanned by an orthonormal basis $\{v_j\}_{j=1}^r$, and let $\{\mu_j\}_{j=1}^k \subset \mathcal{P}$ be a finite set of parameter samples. Define the loss function 
$$\mathcal{L}(\widehat{O}) = \sum_{i=1}^k\Vert F(\widehat{O};\hat{x},\theta,\mu_i)-\hat{x}(\mu_i)\Vert_2^2,$$ 
where $F$ is given by $\cref{F}$. Assume:
\begin{enumerate}
    \item[(C1)] For all parameter samples $\mu_i \in \{\mu_j\}_{j=1}^k$, the full-order solution $x(\mu_j)$ satisfies $x(\mu_j) = V_r\hat{x}(\mu_j)$, where $V_r = [v_1, \ldots, v_r]$. That is, the snapshot reconstruction is exact at the sample parameters. 
    \item[(C2)] The matrix $D \in \mathbb{R}^{k \times p(r)}\ (k>p(r))$ has full column rank.
\end{enumerate}
If condition (C1) holds, then the loss function $\mathcal{L}$ has a global minimizer $\widehat{O}$ satisfying $\mathcal{L}(\widehat{O}) = 0$. If condition (C2) also holds, this global minimizer is unique.
\end{theorem}
\begin{proof}
    Assume (C1) holds. Then for each $\mu\in \{\mu_j\}_{j=1}^k$, equation $\eqref{reduce}$ is satisfied exactly with operators $\widehat{C}_{2,j},\ \widehat{C}_{1,j},$ and $ \widehat{C}_{0,j}$ derived from \cref{C0,C1,C2}. 
    Constructing $\widehat{O}$ from these operators as $\eqref{O}$, we have $\hat{x}(\mu)=F(\widehat{O}; \hat{x},\theta,\mu)$
    for all $\mu\in \{\mu_j\}_{j=1}^k.$ Thus, we have $\mathcal{L}(\widehat{O})=0.$ Since $\mathcal{L}$ is non-negative, $\widehat{O}$ is a global minimizer.

    To show uniqueness, suppose $\widehat{O}_1,\ \widehat{O}_2$ are both minimizers with $\mathcal{L}(\widehat{O}_1)=\mathcal{L}(\widehat{O}_2)=0,$ which implies
    \begin{equation}
        F(\widetilde{O};\hat{x},\theta,\mu_i)=F(\widehat{O}_1;\hat{x},\theta,\mu_i)-F(\widehat{O}_2;\hat{x},\theta,\mu_i)=0_r,\ \ i=1,\cdots ,k,
        \label{unique}
    \end{equation}
    where 
    $$
    \widetilde{O}=\widehat{O}_1-\widehat{O}_2=
    \begin{bmatrix}
		\widetilde{C}_{2,1} & 
		\cdots & 
		\widetilde{C}_{2,n_{C_2}} & 
		\widetilde{C}_{1,1} & 
		\cdots & 
		\widetilde{C}_{1,n_{C_1}} &
        \widetilde{C}_{0,1} & 
		\cdots & 
		\widetilde{C}_{0,n_{C_0}}
	\end{bmatrix}^\top.
    $$
    Moreover, $\eqref{unique}$ can be written as
    $$
    D\widetilde{O}=0_{k\times r}.
    $$
    Since $D$ is of full column rank, we have $\widetilde{O}=0_{p(r)\times r}$; consequently, $\widehat{O}_1=\widehat{O}_2$.
\end{proof}

The sufficient conditions in \Cref{theorem3.5} can be difficult to verify in practice. Moreover, deriving practically verifiable sufficient conditions to guide the optimal selection of snapshot parameters $\mu_j$ remains an open and challenging problem. To ensure a stable and unique solution, especially when $D$ is potentially rank-deficient, we introduce a regularization framework. The Tikhonov-regularized problem can be expressed as
\begin{equation}
\min_{\widehat{O} \in \mathbb{R}^{p(r) \times r}} \| D \widehat{O} - \widehat{X} \|^2_F + \|\Lambda \widehat{O} \|^2_F,\ 
\label{regularized}
\nonumber
\end{equation}
where $\Lambda \in \mathbb{R}^{p(r) \times p(r)}$ is a regularization matrix. 
The solution satisfies the modified normal equations:
\begin{equation}
(D^\top D + \Lambda^\top \Lambda) \widehat{O} = D^\top \widehat{X}.
\label{2.8}
\end{equation}

A practical choice for $\Lambda$ is a diagonal matrix \cite{Mcquarrie2023} such that
$$
\|\Lambda \widehat{O} \|^2_F =  \lambda_1^2 \left( \sum_{i=1}^{n_{C_1}} \|\widehat{C}_{1,i} \|^2_F + \sum_{i=1}^{n_{C_0}} \|\widehat{C}_{0,i} \|^2_F \right)+\lambda_2^2 \left( \sum_{i=1}^{n_{C_2}} \|\widehat{C}_{2,i} \|^2_F \right).
$$
The hyper-parameters are selected to minimize the mean squared training error: 
$$
\frac{1}{k} \sum_{i=1}^{k} \|\hat{x}(\mu_i) - \bar{x}(\mu_i)\|^2_2,
$$
where \(\hat{x}(\mu_i)\) is the known training data, and \(\bar{x}(\mu_i)\) is the solution of the ROM $\cref{reduce}$ for hyper-parameters $\lambda_1,\ \lambda_2$.

In the offline phase, the basis $V_r$ is constructed, and the operators $\widehat{O}$ are inferred by solving $\cref{regularized}$. For any $w<r,$ a $w$-dimensional ROM is obtained by restricting to the first $w$ POD modes.
In the online phase, the parametric ROM is solved rapidly to obtain $\hat{x}(\mu)$ for new parameter $\mu$, and the full-state approximation is recovered via $x(\mu)\approx
V_r\hat{x}(\mu).$ 
Thus, the surrogate model of $\eqref{polynomial structure}$ is 
\begin{equation}
    \begin{cases}
        \widehat{C}_2(\mu)\hat{x}^2(\mu)+\widehat{C}_1(\mu)\hat{x}(\mu)+\widehat{C}_0(\mu)=\hat{x}(\mu),\\
        \widetilde{x}(\mu)=V_r\hat{x}(\mu),
    \end{cases}
    \label{surrogate model}
    \nonumber
\end{equation}
where $\widetilde{x}(\mu)$ is approximate solution of $x(\mu).$
\begin{remark}
    Reduced operators for a smaller dimension $w<r$  can be constructed by truncation: $\widehat{C}_0(\mu)$ retains the first $w$ rows; $\widehat{C}_1(\mu)$ retains the first $w$ rows and columns; $\widehat{C}_2(\mu)$ retains the first $w$ rows and first $w(w+1)/2$ columns.
\end{remark}
\section{Numerical experiments}
\label{sec4}
In this section, we present several numerical examples to demonstrate the performance of the OpInf method. In \cref{sec4.1}, the method is applied to continuous-time PALEs. It achieves good accuracy and high computational efficiency, particularly when solutions for a large number of new parameter values are required, outperforming conventional numerical approaches in terms of speed. In \cref{sec4.2}, discrete-time PALEs is considered, further demonstrating the capability of OpInf in handling two-dimensional parametric problems. In \cref{sec4.3}, we investigate continuous-time coupled PALEs. The results show that OpInf delivers accuracy comparable to that of intrusive projection-based reduced-order methods, confirming its effectiveness for coupled systems. Finally, \cref{sec4.4} presents continuous-time PAREs, highlighting the robustness and applicability of the OpInf method for nonlinear equations.

Let $k$ parameters be distributed over $\mathcal{P}$. Denote by $x(\mu_i)$ and $\hat{x}(\mu_i)$ the full-order and reduced-order solutions corresponding to the parameter $\mu_i$, respectively. The average relative error over these $k$ parameters is defined as
$$
er_{\text{avg}} = \frac{1}{k} \sum_{i=1}^{k} \frac{\| x(\mu_i) - V_r\hat{x}(\mu_i) \|_2}{\| x(\mu_i) \|_2}.
$$
\subsection{Continuous-time PALEs}
\label{sec4.1}
The parameter domain is defined as $\mathcal{P} = [0.1, 2]$. We consider the case of $s=1$, in which the linear operator 
$\mathcal{L}_1$ is defined as 
$$
\mathcal{L}_1(X;\mu)=A^\top(\mu)X(\mu)+X(\mu)A(\mu).
$$
In this setting, equation $\eqref{1}$ reduces to a continuous-time PALE,
where
$$A(\mu)=\frac{1}{\mu}A_1+\frac{1}{\mu^2}A_2,\ M(\mu)=\frac{1}{\mu}M_1+M_2,
$$
$$
M_1=
\begin{bmatrix}
0.2 & 0 & \cdots & 0 & 0.2
\end{bmatrix},\
M_2=
\begin{bmatrix}
0 & 0.2 & \cdots & 0.2 & 0
\end{bmatrix},
$$
$$
A_1=
\begin{bmatrix}
5 & 0.3 & 0.3\\
& 5 & 0.3 & 0.3\\
& & \ddots & \ddots & \ddots\\
& & & 5 & 0.3 & 0.3\\
& & & & 5 & 0.3\\
& & & & & 5\\
\end{bmatrix},\ A_2=
\begin{bmatrix}
0\\
0.2 & 0\\
0.2 & 0.2 & 0\\
&  \ddots & \ddots & \ddots\\
& & 0.2 & 0.2 & 0\\
& & & 0.2 & 0.2 & 0\\
\end{bmatrix}.
$$

To evaluate the performance of the ROM, we set $r = 8$ and analyze the results from various angles.
\Cref{figure1} compares the relative errors of different basis selection strategies and examines the effect of the number of training parameters on POD accuracy. The left subplot clearly demonstrates the superior accuracy of the POD-based approach, while the right subplot shows how the model error decreases as more training parameters are incorporated, indicating favorable convergence behavior.
\begin{figure}
    \centering
    \begin{subfigure}{0.43\textwidth} 
        \includegraphics[width=\linewidth]{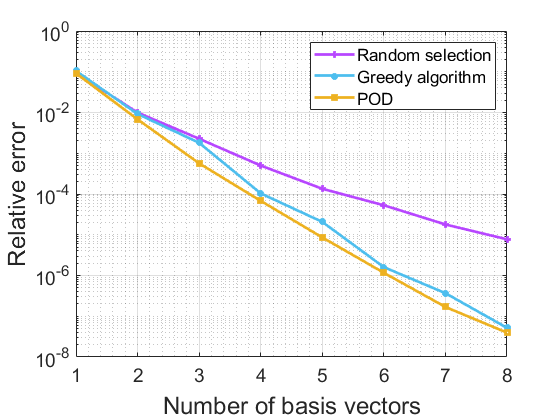}
        \label{basis method}
    \end{subfigure}
    \hspace{-0.5cm}
    \begin{subfigure}{0.43\textwidth} 
        \includegraphics[width=\linewidth]{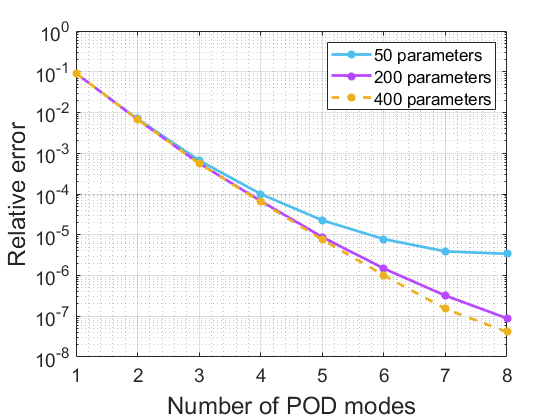}
        \label{number of parameters}
    \end{subfigure}
    \caption{Relative errors for different basis selection methods (left) and for the POD method with varying size of training parameters (right), evaluated at the full-order system dimension of $N=1024^2$ over $10^4$ uniformly distributed random test parameters.}
    \label{figure1}
\end{figure}

\begin{figure}
    \centering
    \begin{subfigure}[b]{0.32\textwidth}
        \includegraphics[width=\textwidth]{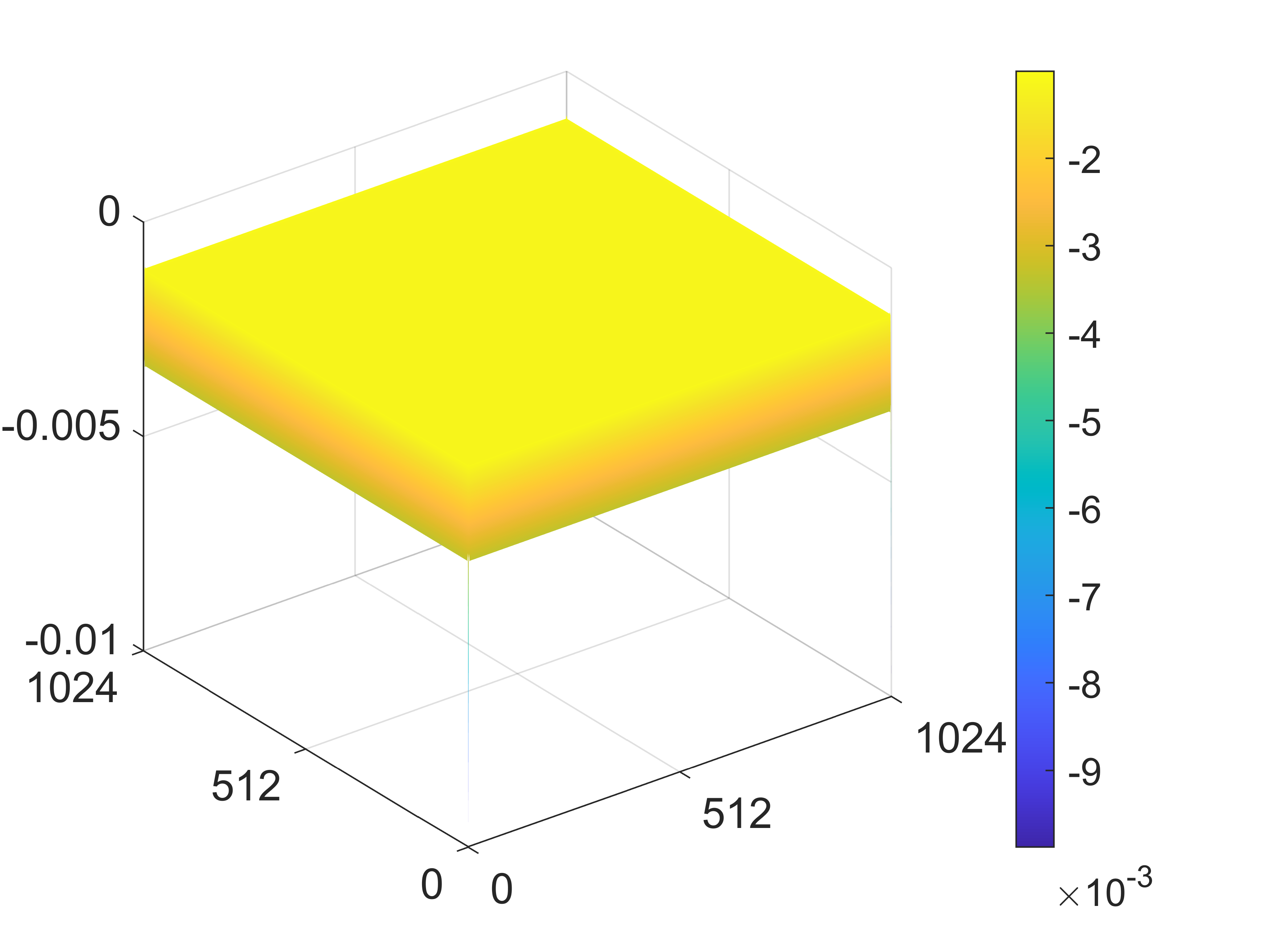}
    \end{subfigure}
    \hspace{-0.5cm}
    \begin{subfigure}[b]{0.32\textwidth}
        \includegraphics[width=\textwidth]{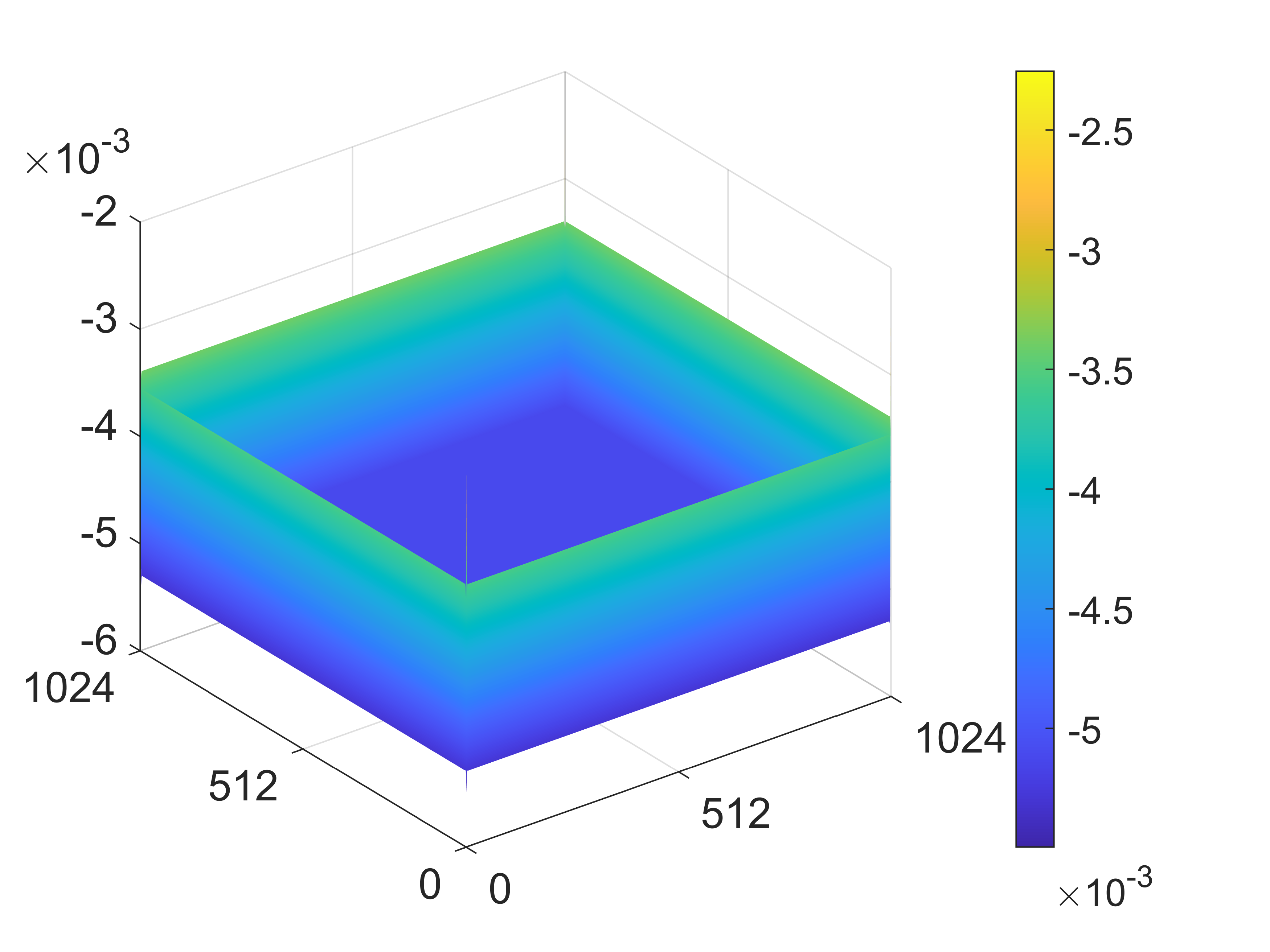}
    \end{subfigure}
    \hspace{-0.5cm}
    \begin{subfigure}[b]{0.32\textwidth}
        \includegraphics[width=\textwidth]{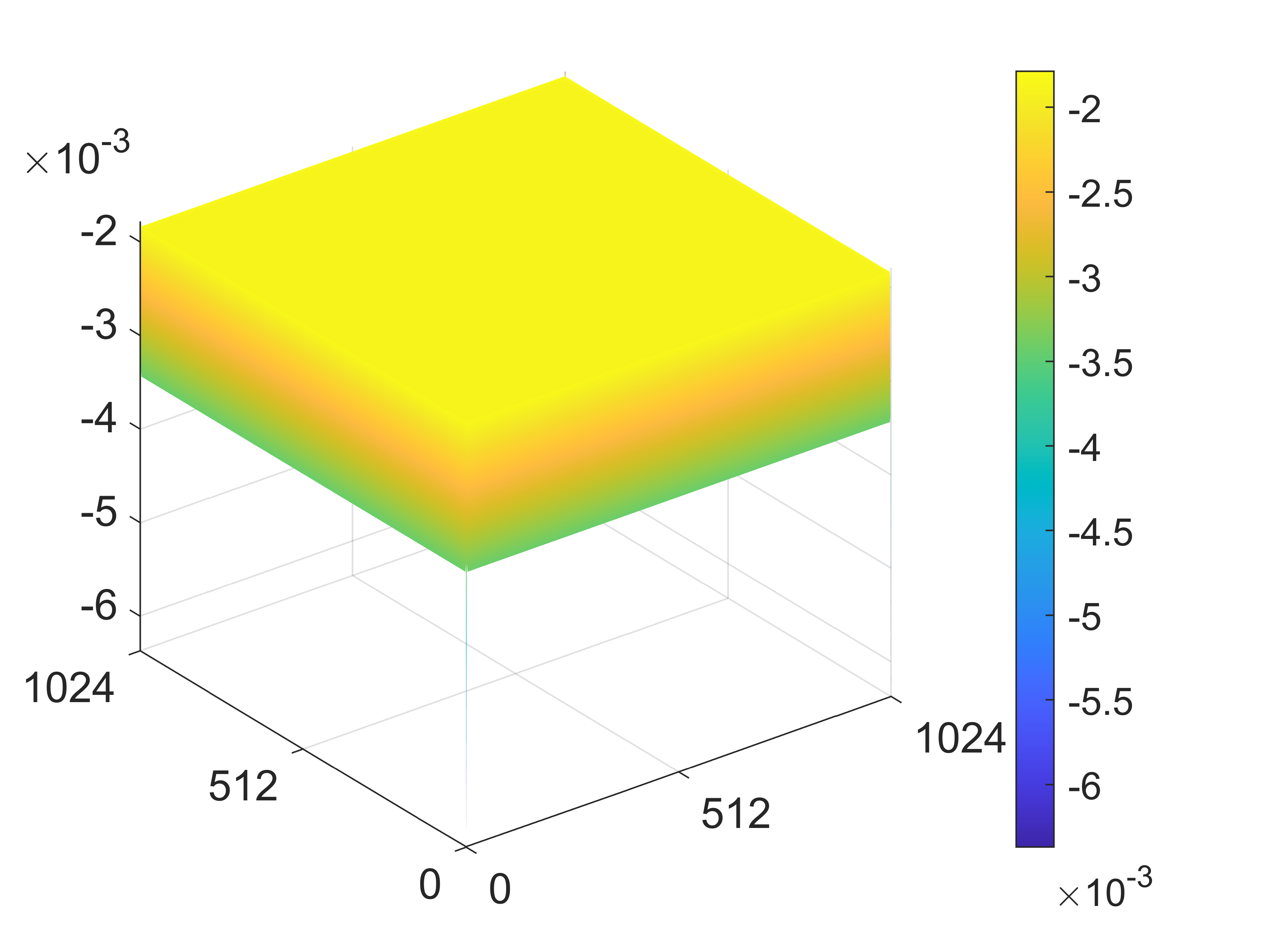}
    \end{subfigure}
     \vspace{-0.1cm} 
    \begin{subfigure}[b]{0.32\textwidth}
        \includegraphics[width=\textwidth]{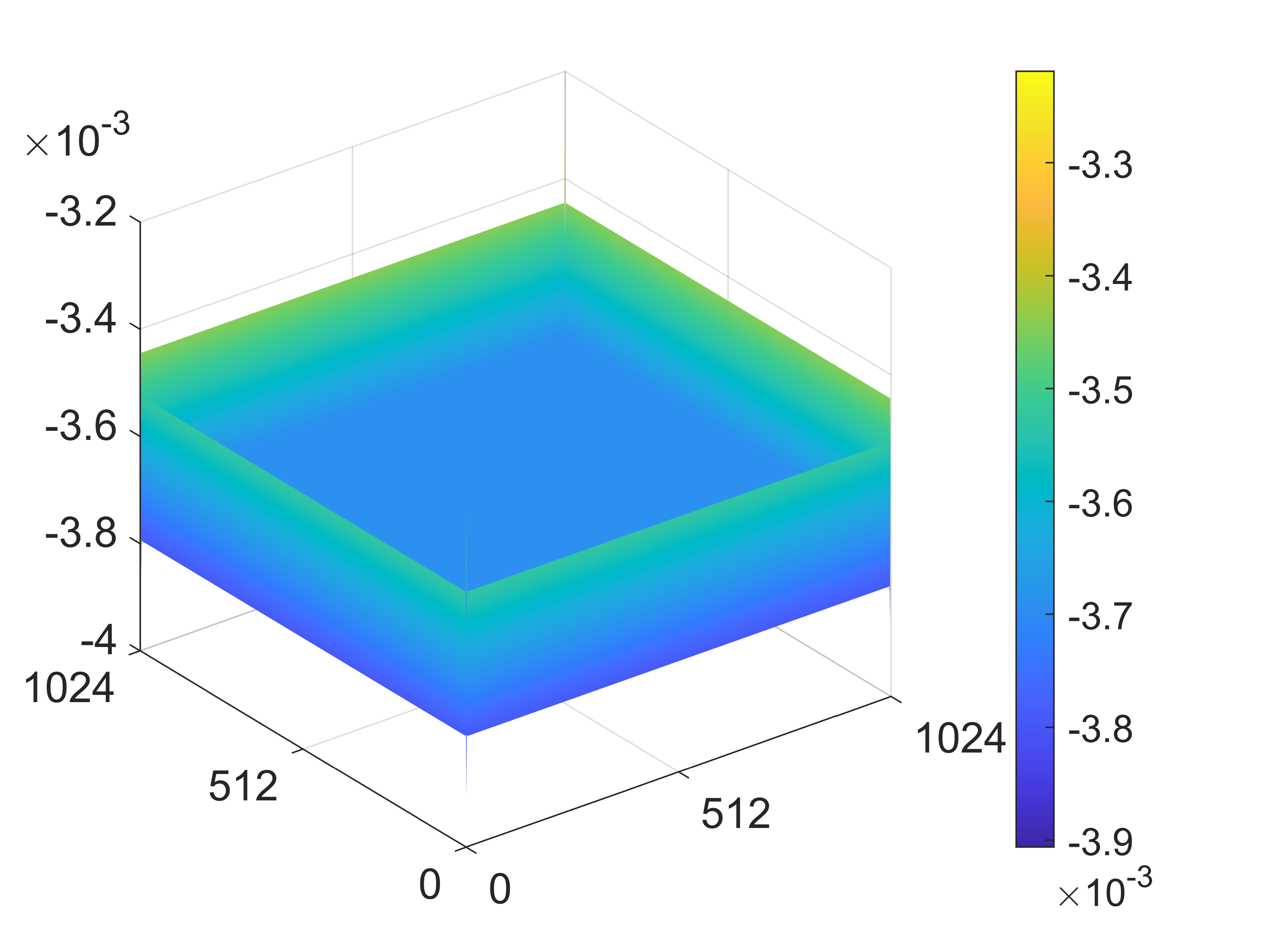}
    \end{subfigure}
    \hspace{-0.5cm}
    \begin{subfigure}[b]{0.32\textwidth}
        \includegraphics[width=\textwidth]{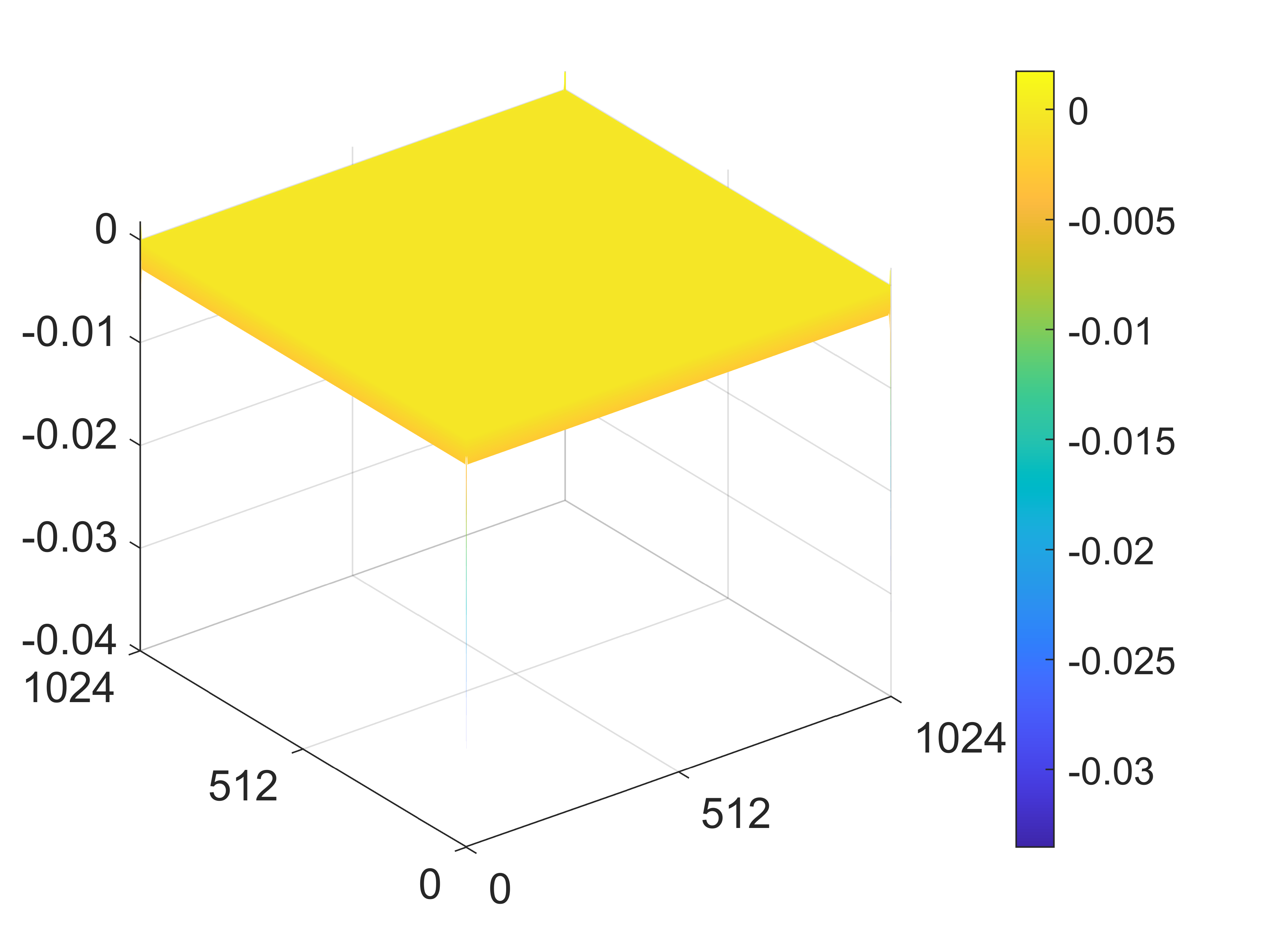}
    \end{subfigure}
    \hspace{-0.5cm}
    \begin{subfigure}[b]{0.32\textwidth}
        \includegraphics[width=\textwidth]{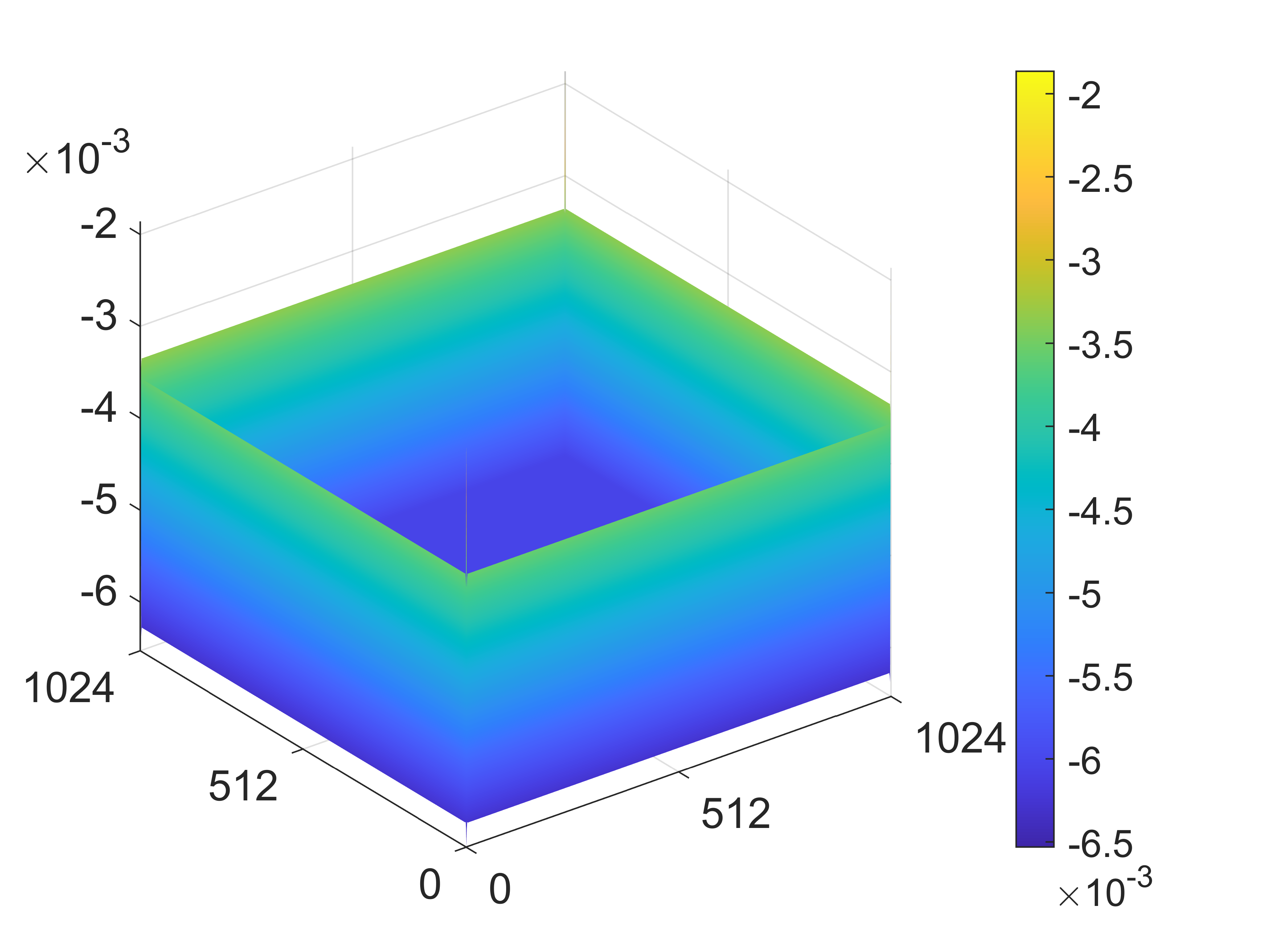}
    \end{subfigure}
    \caption{The snapshots, evaluated at the full-order system dimension of $N=1024^2$.}
    \label{snapshot}
\end{figure}
\Cref{tab:comparison1} reports the average relative errors for different sizes of POD modes and provides a quantitative comparison of computational efficiency between the OpInf method and the MATLAB lyap solver. We record the CPU time(offline ($T_{\text{off}}$), online ($T_{\text{on}}$), total ($T_{\text{tot}}$), and average online ($T_{\text{avg}}$) time) required to compute solutions for $10^4$ uniformly distributed test parameters. The results show that OpInf achieves a speed-up of over three orders of magnitude in average solution time while maintaining accuracy. \Cref{snapshot} displays six solution snapshots, each corresponding to a specific  parameter sample, illustrating the influence of parameter variations on the system  behavior. Furthermore, \Cref{POD mode} presents the first 6 POD modes, providing visual  insight into the dominant structures captured by the reduced basis.
Finally, \Cref{figure2} and \Cref{figure3}  illustrate the performance of models trained with 200 and 50 parameters, respectively, both tested on the same set of 1000 random parameters. Interestingly, the model trained with fewer but well-chosen parameters achieves a lower average relative error, suggesting that simply increasing the number of training samples does not guaranty better accuracy. Instead, these results highlight the importance of selecting representative training parameters that effectively capture essential features of the parameter space.
\begin{figure}[htbp]
    \centering
    \begin{subfigure}[b]{0.32\textwidth}
        \includegraphics[width=\textwidth]{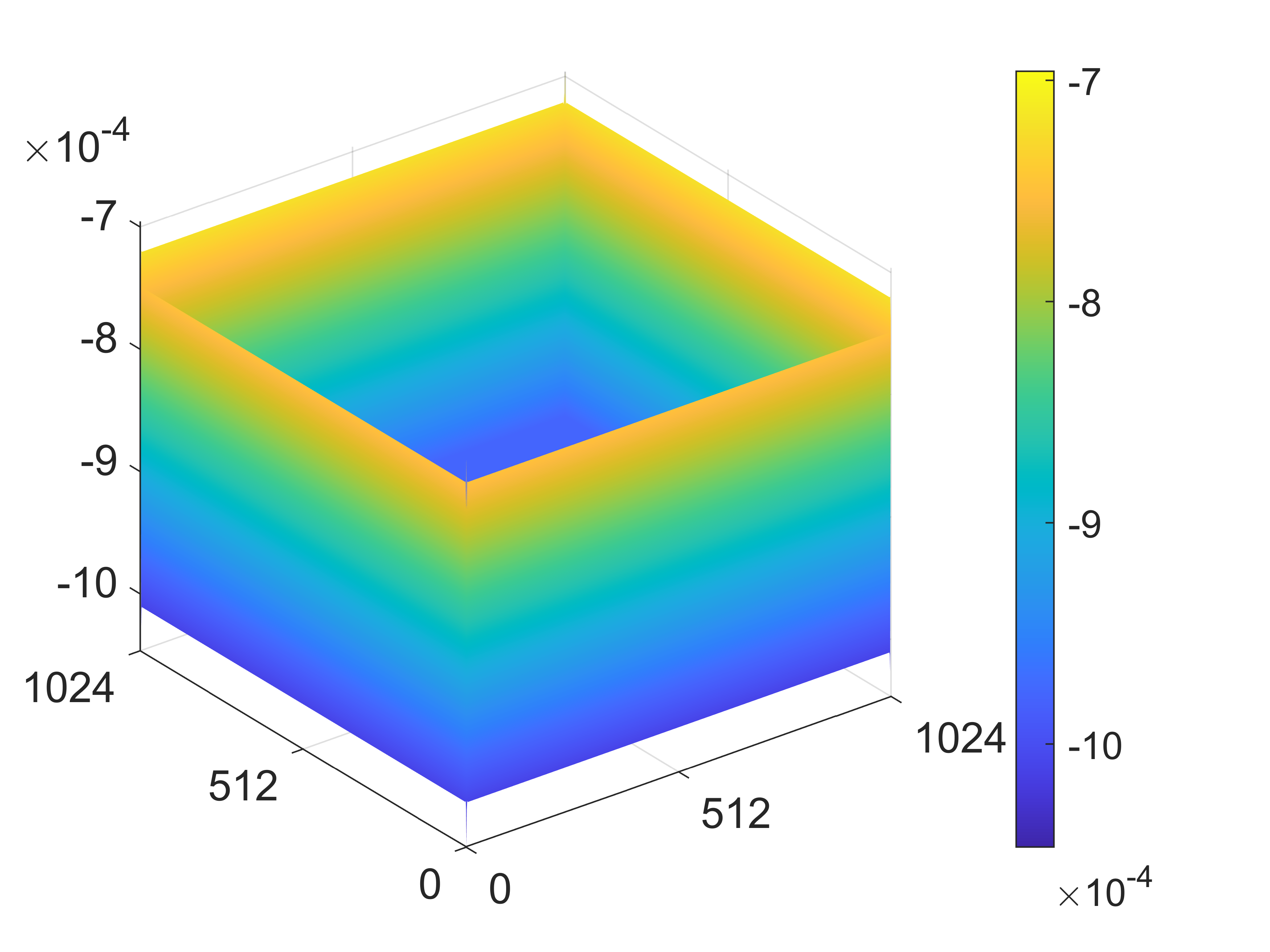}
    \end{subfigure}
    \hspace{-0.5cm}
    \begin{subfigure}[b]{0.32\textwidth}
        \includegraphics[width=\textwidth]{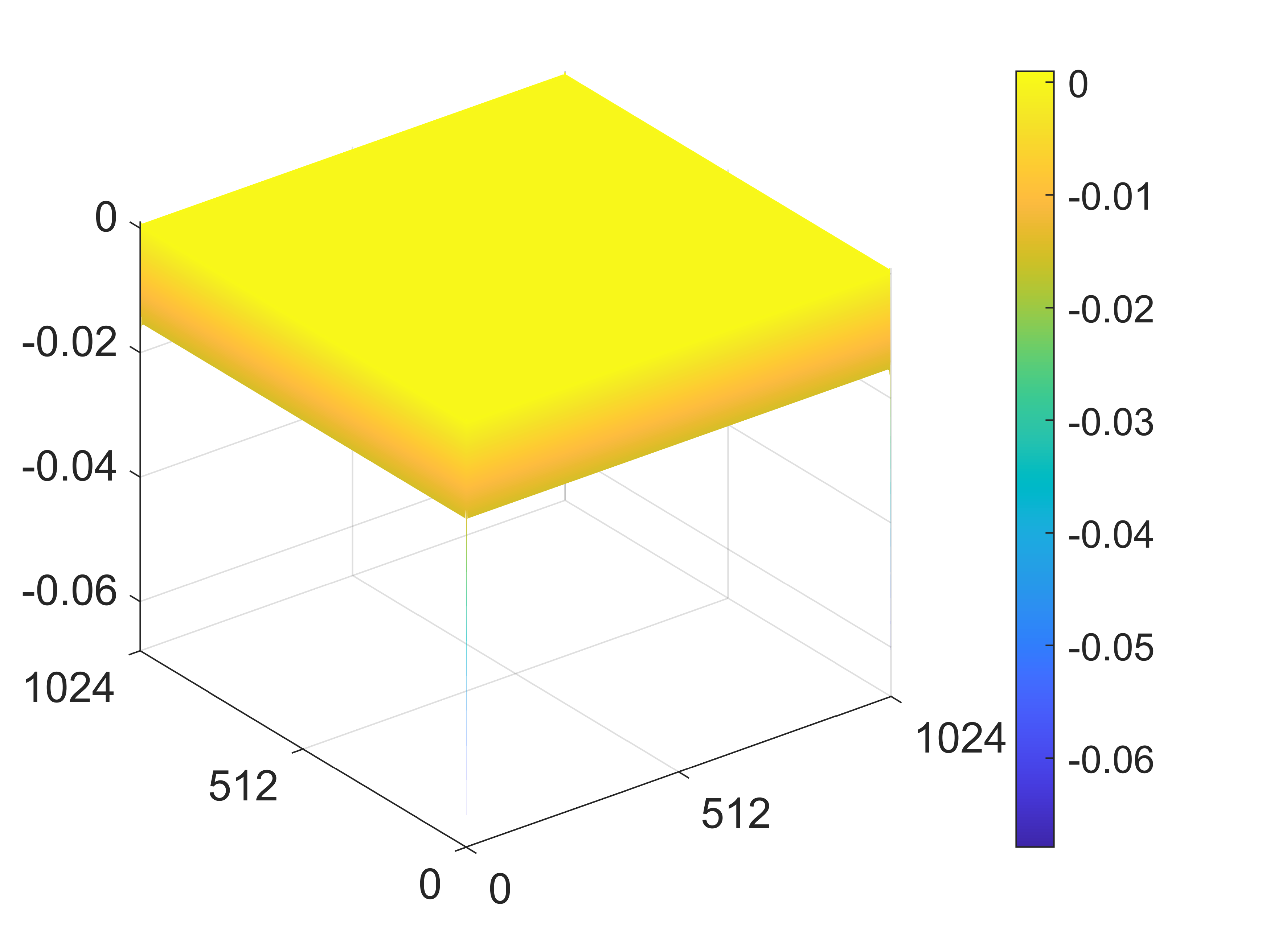}
    \end{subfigure}
    \hspace{-0.5cm}
    \begin{subfigure}[b]{0.32\textwidth}
        \includegraphics[width=\textwidth]{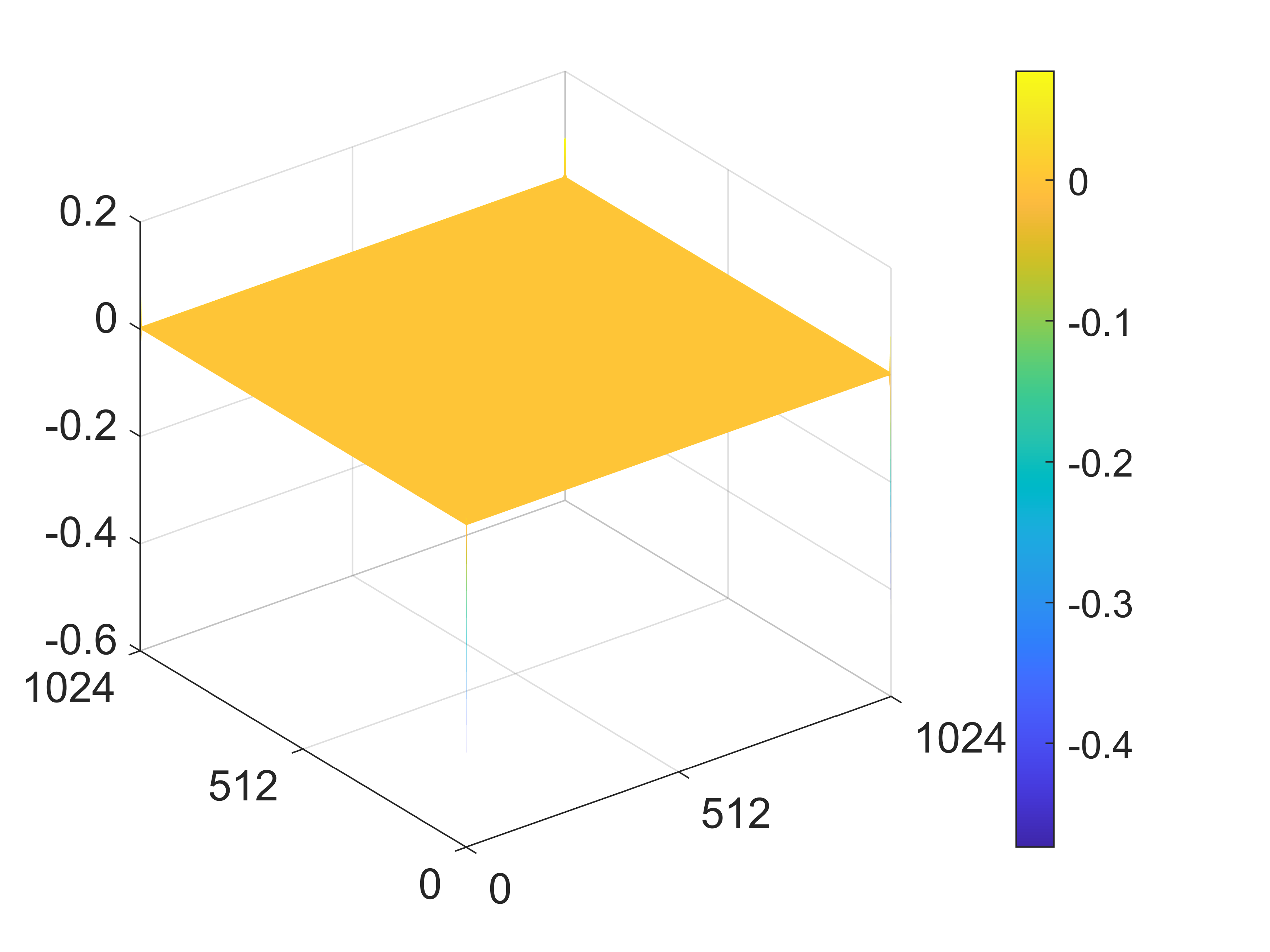}
    \end{subfigure}
     \vspace{0.1cm} 
    \begin{subfigure}[b]{0.32\textwidth}
        \includegraphics[width=\textwidth]{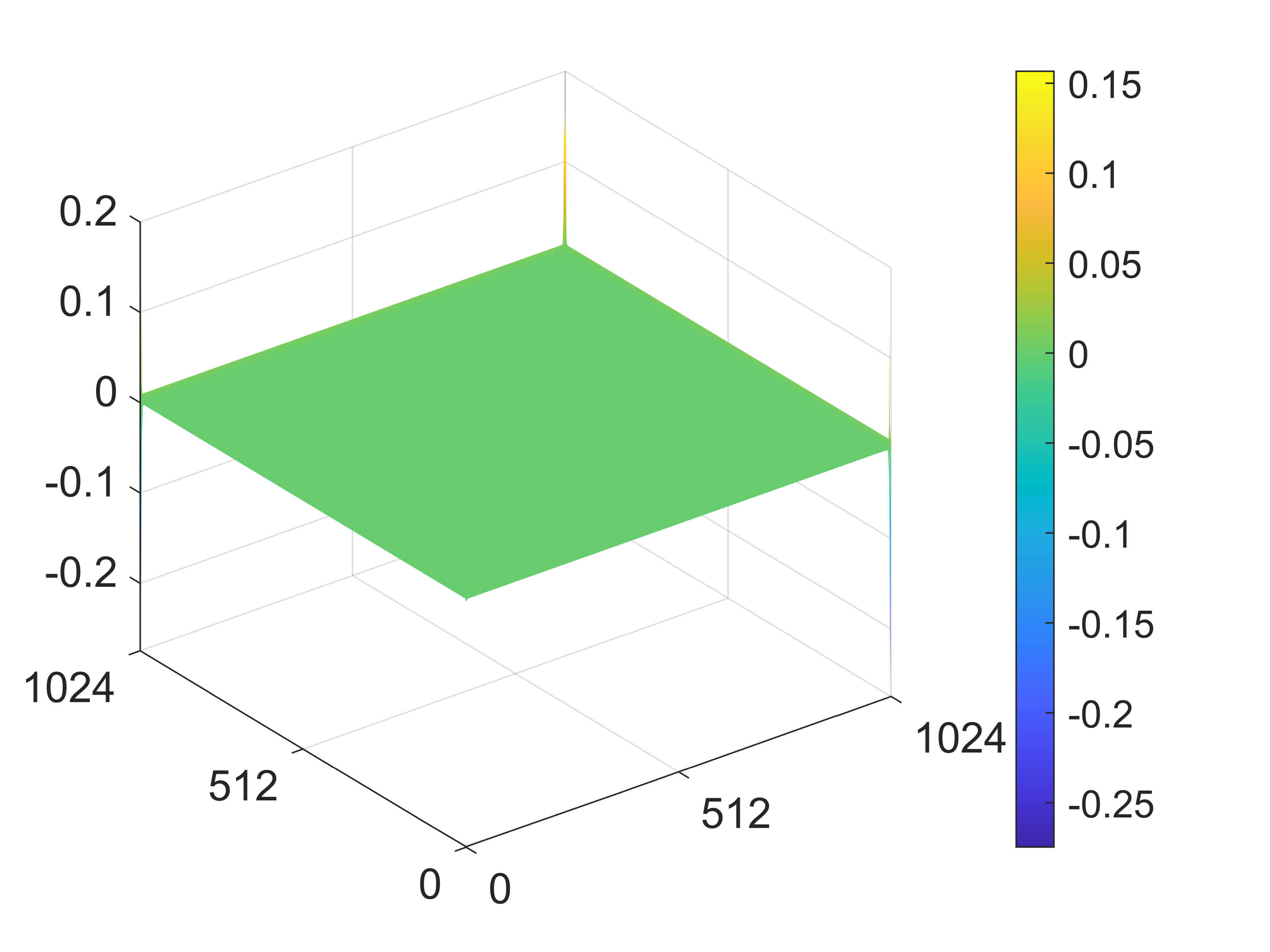}
    \end{subfigure}
    \hspace{-0.5cm}
    \begin{subfigure}[b]{0.32\textwidth}
        \includegraphics[width=\textwidth]{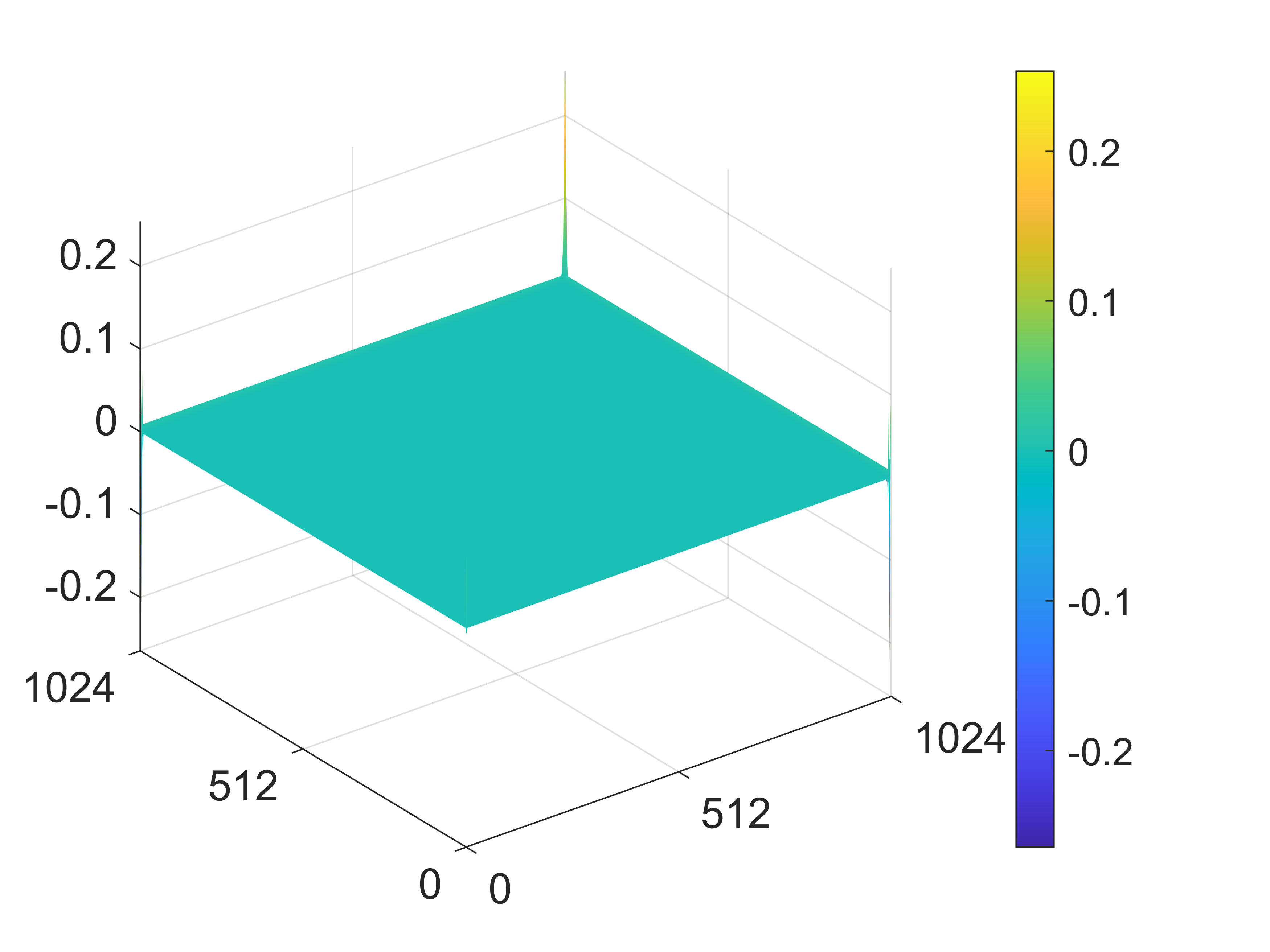}
    \end{subfigure}
    \hspace{-0.5cm}
    \begin{subfigure}[b]{0.32\textwidth}
        \includegraphics[width=\textwidth]{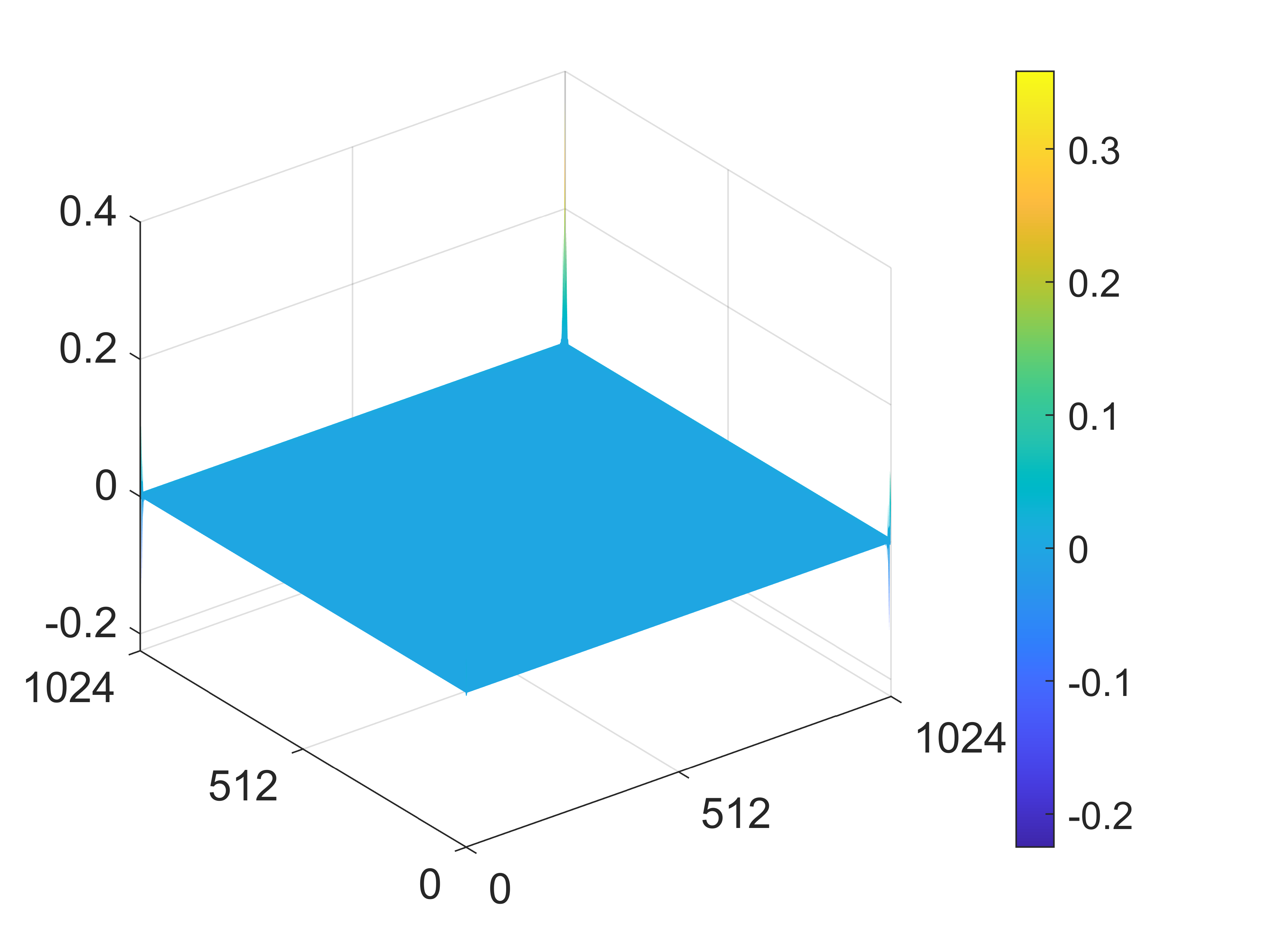}
    \end{subfigure}
    \caption{The POD modes, evaluated at the full-order system dimension of $N=1024^2$.}
    \label{POD mode}
\end{figure}
\begin{table}[htbp]
\centering
\caption{The average relative error and a comparison of CPU time between the MATLAB lyap solver (Reference method) and the OpInf, evaluated at the full-order system dimension of $N=1024^2$ over $10^4$ uniformly distributed random test parameters.}
\label{tab:comparison1}

\begin{tabular}{cccc}
\toprule[1.5pt]
\multirow{2}{*}{\textbf{Strategies}} & 
\multicolumn{2}{c}{\textbf{The OpInf method}} & 
\multirow{2}{*}{\textbf{Reference method}} \\
\cmidrule{2-3} 
& \textbf{\(r = 4\)} & \textbf{\(r = 8\)} & \\ 
\midrule 
\(er_{avg}\) & \(6.77\times 10^{-5}\) & \(3.96\times 10^{-8}\)\\
\midrule
\(T_{\text{off}}\) & 122.44\,s & 122.44\,s & - \\
\midrule
\(T_{\text{on}}\) & 2.04\,s & 2.08\,s & - \\
\midrule
\(T_{\text{tot}}\) & 124.48\,s & 124.52\,s & \(7.01 \times 10^{3}\,s\) \\
\midrule
\(T_{\text{avg}}\)  & \(2.04 \times 10^{-4}\,s\) & \(2.08 \times 10^{-4}\,s\) & \(7.01 \times 10^{-1}\,s\) \\
\bottomrule[1.5pt] 
\end{tabular}
\end{table}
\begin{table}[htbp]
\centering
\caption{The average relative error and a comparison of CPU time between the MATLAB dlyap solver (Reference method) and the OpInf, evaluated at the full-order system dimension of $N=512^2$ over $10^4$ uniformly distributed random test parameters.}
\label{tab:comparison}
\begin{tabular}{cccc}
\toprule[1.5pt]
\multirow{2}{*}{\textbf{Strategies}} & 
\multicolumn{2}{c}{\textbf{The OpInf method}} & 
\multirow{2}{*}{\textbf{Reference method}} \\
\cmidrule{2-3} 
& \textbf{\(r = 3\)} & \textbf{\(r = 6\)} & \\ 
\midrule 
\(er_{avg}\) & \(1.33\times 10^{-3}\) & \(3.04\times 10^{-5}\)\\
\midrule
\(T_{\text{off}}\) & 22.46\,s & 22.46\,s & - \\
\midrule
\(T_{\text{on}}\) & 0.59\,s & 0.79\,s & - \\
\midrule
\(T_{\text{tot}}\) & 23.05\,s & 23.25\,s & \(2.64 \times 10^{3}\,s\) \\
\midrule
\(T_{\text{avg}}\)  & \(0.59 \times 10^{-4}\,s\) & \(0.79 \times 10^{-4}\,s\) & \(2.64 \times 10^{-1}\,s\) \\
\bottomrule[1.5pt] 
\end{tabular}
\end{table}
\begin{figure}
    \centering
    \begin{subfigure}{0.43\textwidth} 
        \includegraphics[width=\linewidth]{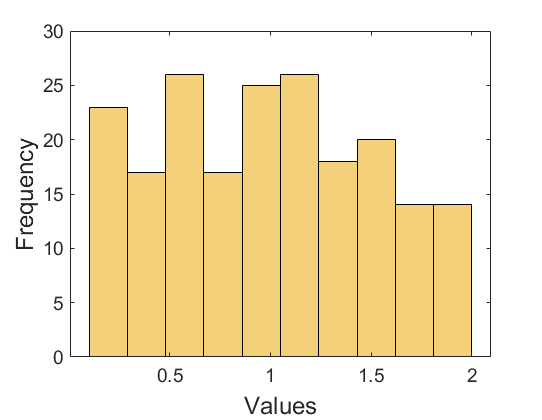}
        \label{rng(2)_1}
    \end{subfigure}
    \hspace{-0.5cm}
    \begin{subfigure}{0.43\textwidth} 
        \includegraphics[width=\linewidth]{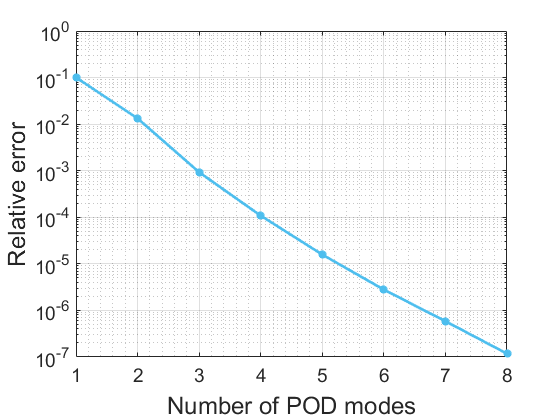}
        \label{rng(2)_2.png}
    \end{subfigure}
    \caption{Distribution of the 200 training parameters (left) and the average relative error of the model trained on them, evaluated at the full-order system dimension of $N=512^2$ over $1000$ uniformly distributed random test parameters (right). }
    \label{figure2}
\end{figure}
\begin{figure}
    \centering
    \begin{subfigure}{0.43\textwidth} 
        \includegraphics[width=\linewidth]{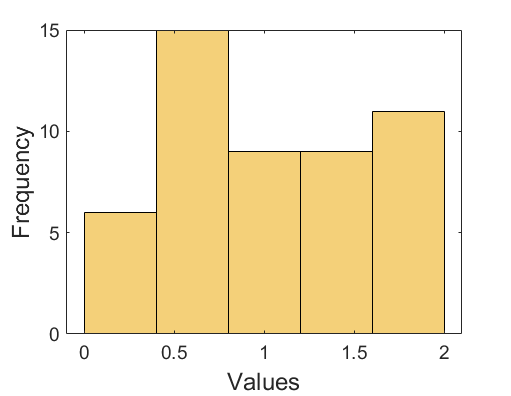}
        \label{rng(5)_1}
    \end{subfigure}
    \hspace{-0.5cm}
    \begin{subfigure}{0.43\textwidth} 
        \includegraphics[width=\linewidth]{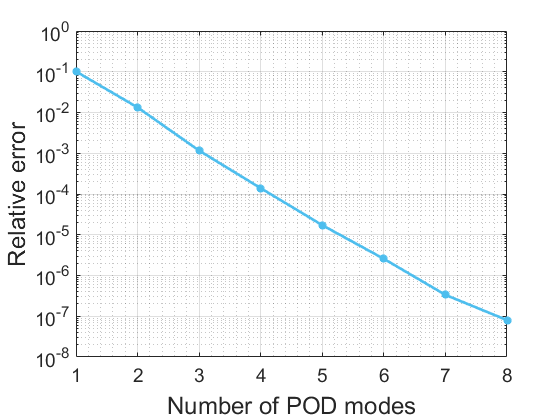}
        \label{rng(5)_2}
    \end{subfigure}
    \caption{Distribution of the 50 training parameters (left) and the average relative error of the model trained on them, evaluated at the full-order system dimension of $N=512^2$ over $1000$ uniformly distributed random test parameters (right). }
    \label{figure3}
\end{figure}
\subsection{Discrete-time PALEs}
\label{sec4.2}
The parameter domain is defined as $\mathcal{P}=[2,6]^2\subset\mathbb{R}^2$. We consider the case of $s=1$, in which the linear operator 
$\mathcal{L}_1$ is defined as 
$$
\mathcal{L}_1(X;\mu)=A^\top(\mu)X(\mu)A(\mu)-X(\mu).
$$
In this setting, equation $\eqref{1}$ reduces to a discrete-time PALE,
where
$$A(\mu)=\mu_1A_1+\frac{1}{\mu_1}A_2+\mu_2A_3,\ M(\mu)=\mu_1M_1+M_2,
$$
$$
A_1=
\begin{bmatrix}
15 \\
& \ddots \\
& & 15 \\
\end{bmatrix},\
M_1=
\begin{bmatrix}
0.1 & \cdots & 0.1 & -0.9
\end{bmatrix},\
M_2=
\begin{bmatrix}
0  & \cdots & 0 & 1
\end{bmatrix},
$$
$$
A_2=
\begin{bmatrix}
0 & 1 \\
0.5 & 0 & 1\\
& \ddots & \ddots & \ddots\\
& & 0.5 & 0 & 1\\
& & & 0.5 & 0
\end{bmatrix},\
A_3=
\begin{bmatrix}
0 & 0 & 1 \\
0 & 0 & 0 & 1\\
0.5 & 0 & 0 & 0 & 1\\
& \ddots & \ddots & \ddots & \ddots & \ddots\\
& & 0.5 & 0 & 0 & 0 & 1\\
& & & 0.5 & 0 & 0 & 0\\
& & & & 0.5 & 0 & 0
\end{bmatrix}.\ 
$$
\begin{figure}
    \centering
    \begin{subfigure}{0.43\textwidth} 
        \includegraphics[width=\linewidth]{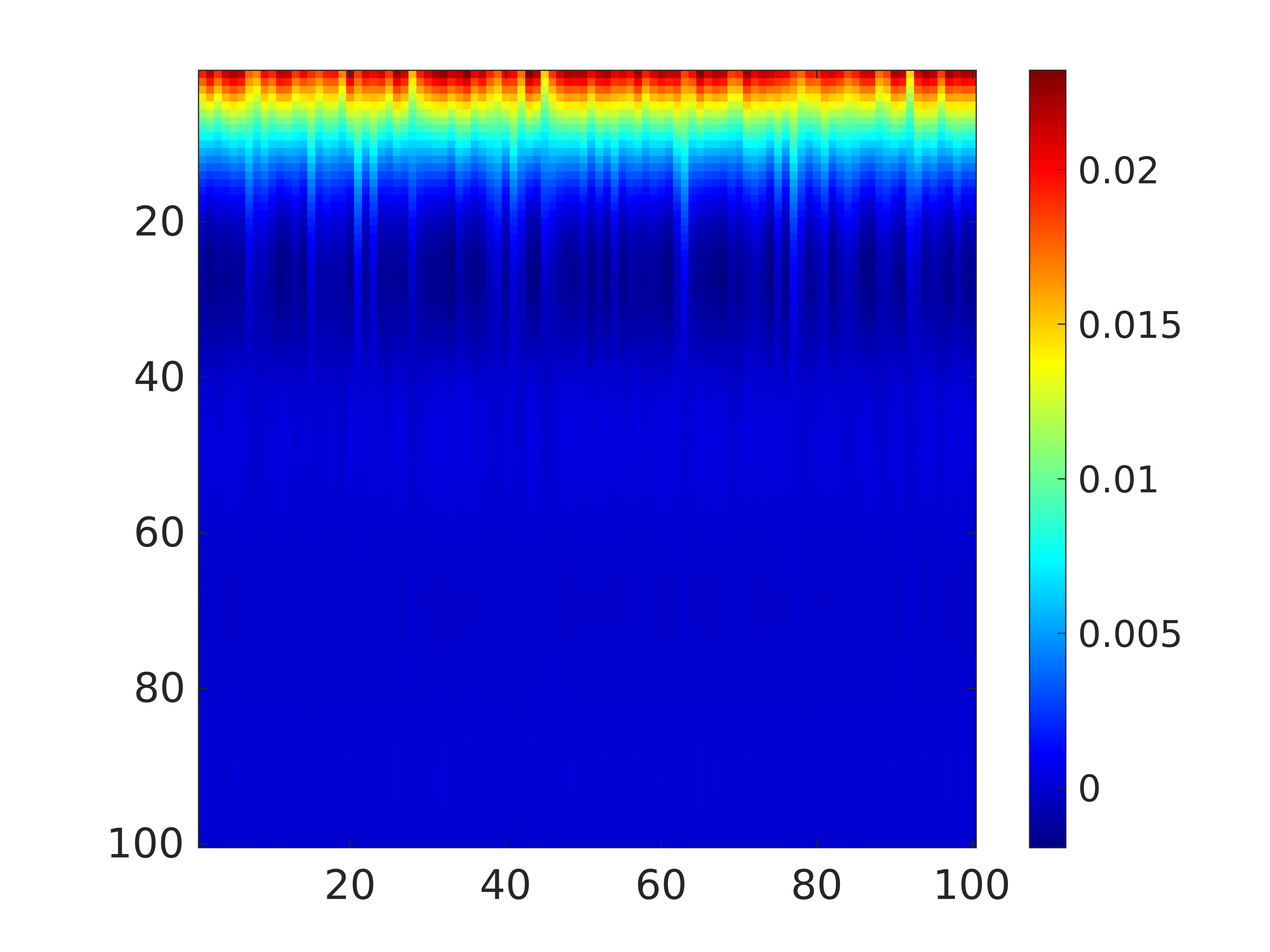}
        \caption{Projection}
        \label{Projection}
    \end{subfigure}
    \hspace{-0.5cm}
    \begin{subfigure}{0.43\textwidth} 
        \includegraphics[width=\linewidth]{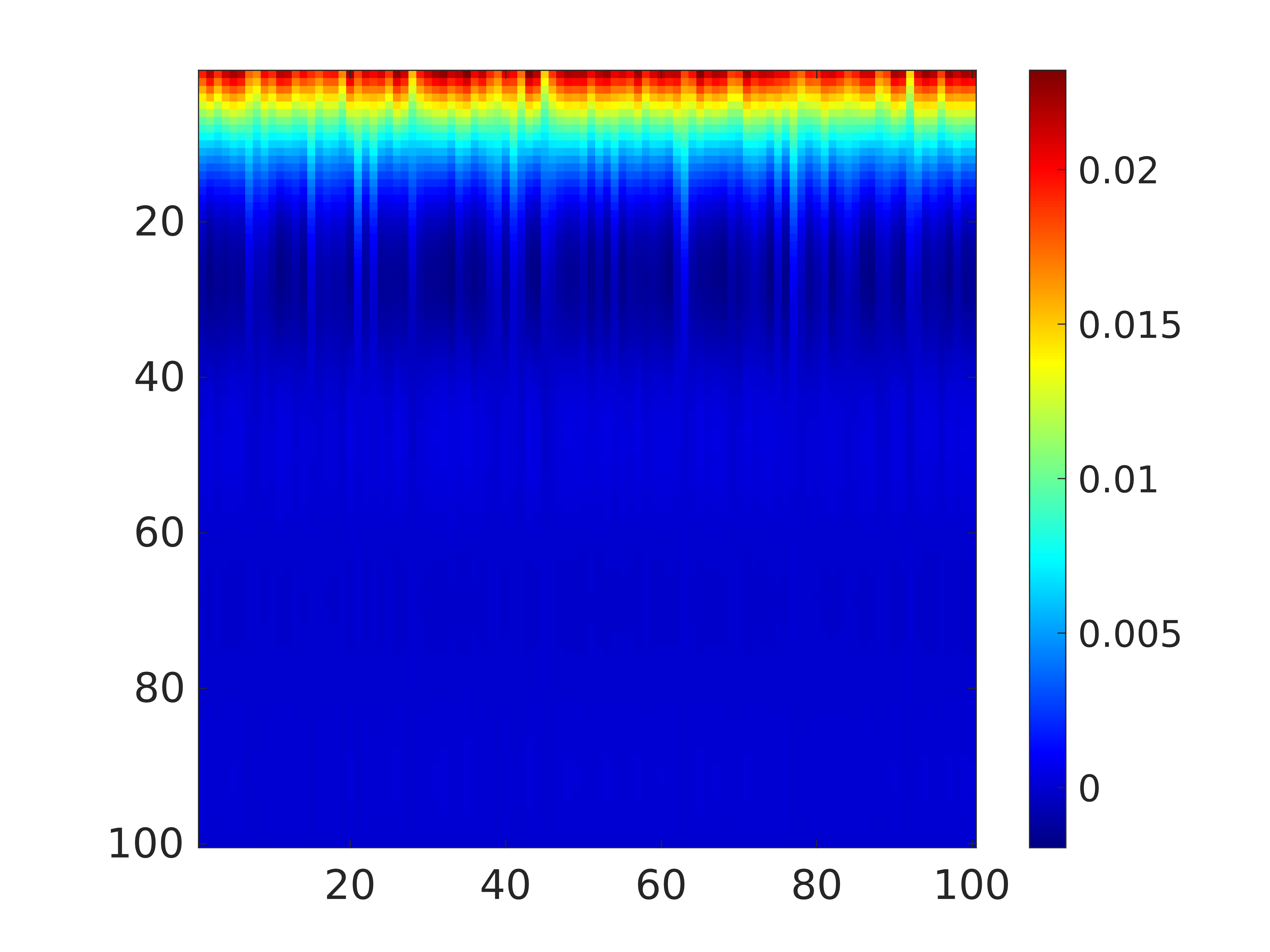}
        \caption{OpInf}
        \label{OpInf}
    \end{subfigure}
    \caption{Full-order solution projected onto the 6-mode POD basis (Left) and the reduced-order solution derived via OpInf on the same basis (Right), evaluated at the full-order system dimension of $N=512^2$ over $10^4$ uniformly distributed random test parameters. }
    \label{Project-OpInf}
\end{figure}

To evaluate the performance of the ROM, we set $r = 6$ and analyze the results from various angles. \Cref{Project-OpInf} provides a visual comparison between the projection of the full-order solution and the reduced-order solution reconstructed via the OpInf method on the same basis. The visual results show that both are highly consistent in morphology and numerical values.

\Cref{figure5} shows the distributions of different-sized training parameters, where each larger set encompasses all parameters from the smaller ones (i.e., the 80-parameter set includes the 40-parameter set, and the 200-parameter set includes the 80-parameter set), and the relative error performance of the POD method with varying numbers of training parameters when $n=512$. 
\begin{figure}
    \centering
    \begin{subfigure}[b]{0.43\textwidth}
        \centering
        \includegraphics[width=\linewidth]{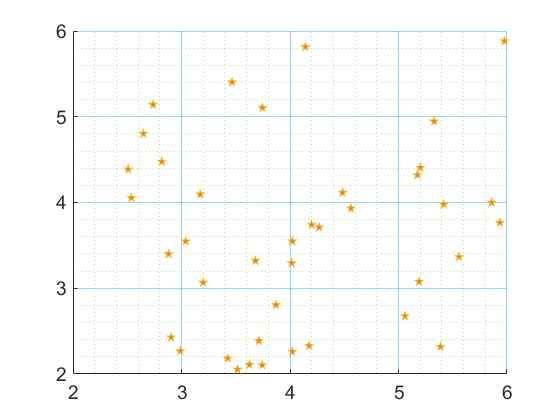}
        \caption{40 training parameters}
    \end{subfigure}
    \hspace{-0.5cm}
    \begin{subfigure}[b]{0.43\textwidth}
        \centering
        \includegraphics[width=\linewidth]{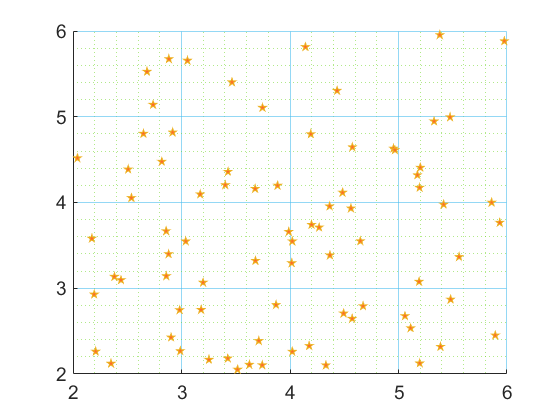}
        \caption{80 training parameters}
    \end{subfigure}
    
    \vspace{0.1cm} 
    
    \begin{subfigure}{0.43\textwidth}
        \centering
        \includegraphics[width=\linewidth]{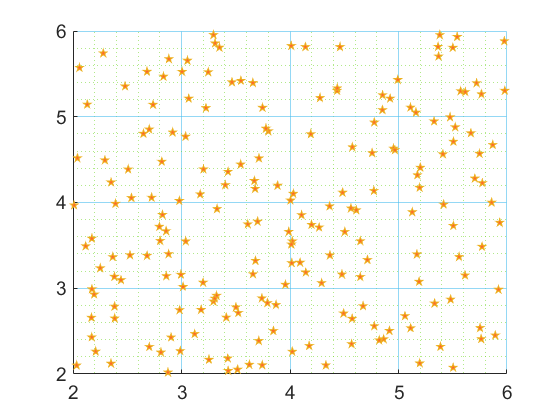}
        \caption{200 training parameters}
    \end{subfigure}
    \hspace{-0.5cm}
    \begin{subfigure}{0.43\textwidth}
        \centering
        \includegraphics[width=\linewidth]{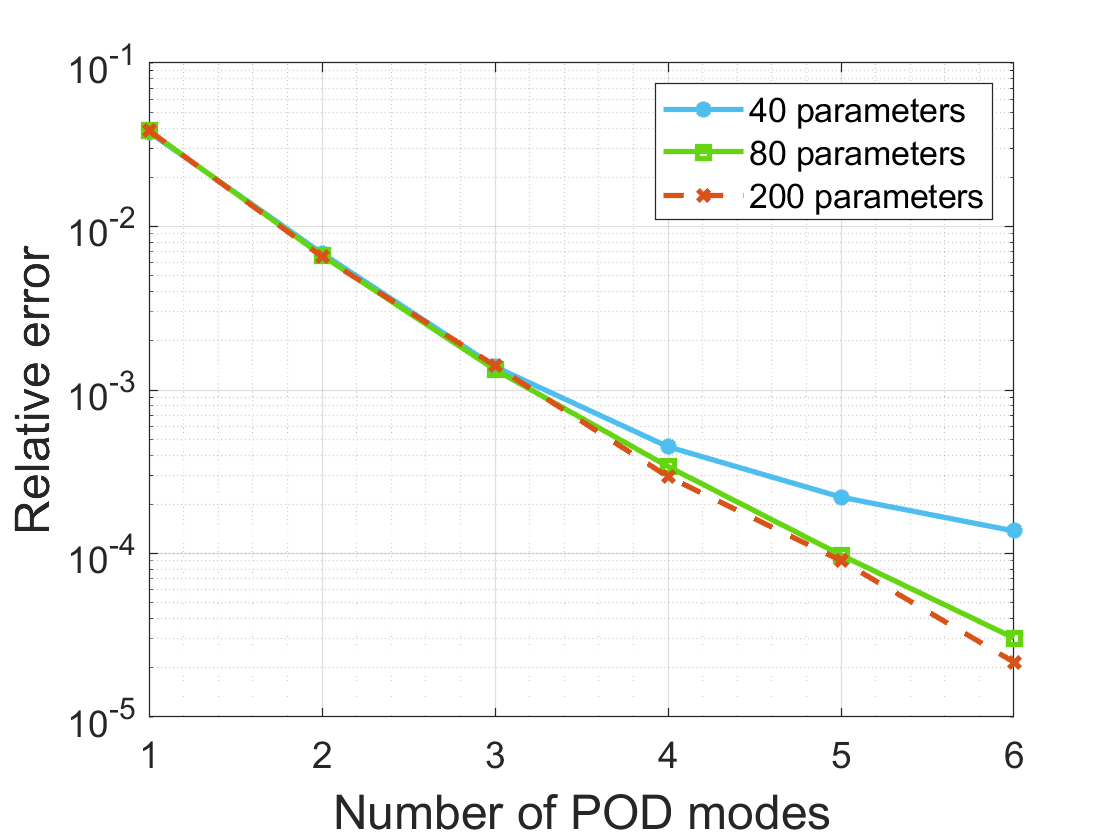}
        \caption{Average relative error}
    \end{subfigure}
    \caption{ Training parameters distributions and the average relative error of the model trained on them, evaluated at the full-order system dimension of $N=512^2$ over $10^4$ uniformly distributed random test parameters.}
    \label{figure5}
\end{figure}
Finally, \Cref{tab:comparison} provides average relative errors with different-sized POD modes and the computational time between the full-order reference method and the OpInf method, clearly demonstrating the high efficiency of the OpInf method in handling large-scale parameter samples. 

\subsection{Continuous-time coupled PALEs}
\label{sec4.3}
The parameter domain is defined as $\mathcal{P}=[1,2]$. We consider the case of $s=2$, in which the linear operators 
are defined as 
$$
\mathcal{L}_1(X_1,X_2;\mu)=A_1^\top(\mu)X_1(\mu)+X_1(\mu)A_1(\mu)-X_1(\mu)+X_2(\mu),
$$
$$
\mathcal{L}_2(X_1,X_2;\mu)=A_2^\top(\mu)X_2(\mu)+X_2(\mu)A_2(\mu)+X_1(\mu)-X_2(\mu).
$$
In this setting, equation $\eqref{1}$ reduces to two continuous-time coupled PALEs,
where
$$A_1(\mu)=\frac{1}{\mu}A_{11}+\frac{1}{\mu^2}A_{12},\ A_2(\mu)=\mu A_{21}+\mu^2A_{22},
$$
$$
M_1(\mu)=\mu M_{11}+M_{12},\ M_2(\mu)=\frac{1}{\mu}M_{21}+M_{22},
$$
$$
M_{11}=
\begin{bmatrix}
    1 & 0 & \cdots & 0
\end{bmatrix},\
M_{12}=
\begin{bmatrix}
    0 & 0 & \cdots & 1
\end{bmatrix},\ M_{21}=2M_{11},\ M_{22}=2M_{12},
$$
$$
A_{11}=
\begin{bmatrix}
    10 \\
    & \ddots\\
    & & 10
\end{bmatrix},\
A_{21}=
\begin{bmatrix}
    8 \\
    & \ddots\\
    & & 8
\end{bmatrix},\
$$
$$
A_{12}=
\begin{bmatrix}
    0 & 2 \\
    3 & 0 & 2\\
    & \ddots & \ddots & \ddots\\
    & & 3 & 0 & 2\\
    & & & 3 & 0
\end{bmatrix},\
A_{22}=
\begin{bmatrix}
    0 & 1 \\
    2 & 0 & 1\\
    & \ddots & \ddots & \ddots\\
    & & 2 & 0 & 1\\
    & & & 2 & 0
\end{bmatrix}.
$$

A comparison of the non-intrusive OpInf method and the intrusive approach is performed using $8$ POD modes, assessing the average relative error over $10$ random test parameters generated in MATLAB. As shown in \Cref{couple1}, the accuracy of OpInf is comparable to that of the intrusive reduced-order modeling method. Furthermore, \Cref{couple2} illustrates the decay of the residual energy with respect to the number of POD modes, as well as the selection of the regularization parameter.
\begin{figure}
    \centering
    \begin{subfigure}{0.43\textwidth} 
        \includegraphics[width=\linewidth]{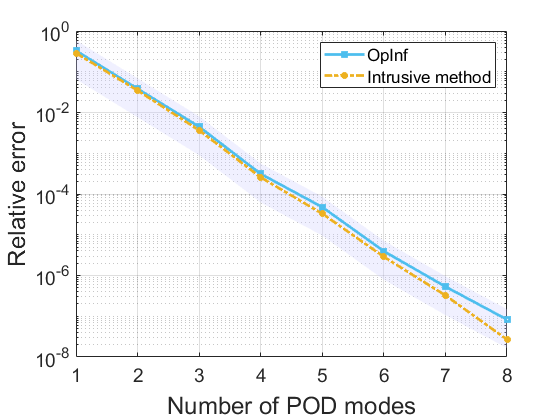}
        \caption{training parameters}
        \label{128couple_train}
    \end{subfigure}
    \hspace{-0.5cm}
    \begin{subfigure}{0.43\textwidth} 
        \includegraphics[width=\linewidth]{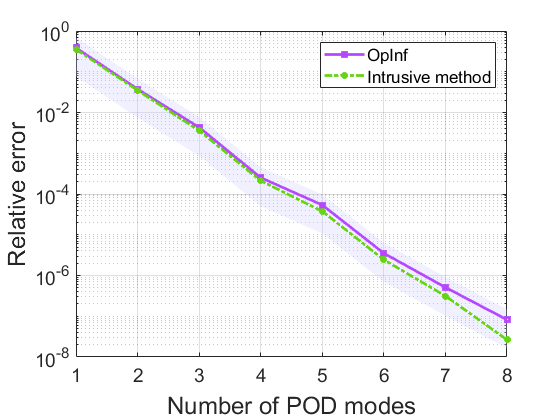}
        \caption{test parameters}
        \label{128couple_test}
    \end{subfigure}
    \caption{Average relative error for OpInf and Intrusive method, evaluated at the full-order system dimension of $N=2\times 128^2$.}
    \label{couple1}
\end{figure}
\begin{figure}
    \centering
    \begin{subfigure}{0.43\textwidth} 
        \includegraphics[width=\linewidth]{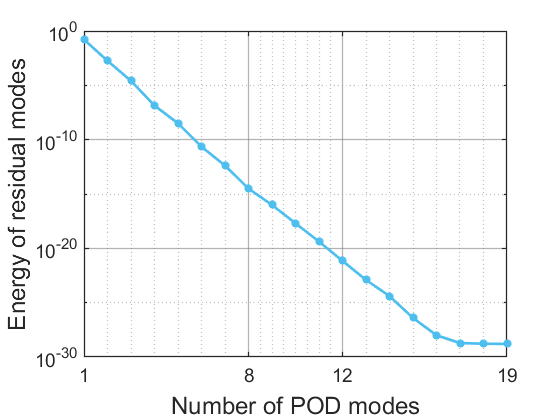}
    \end{subfigure}
    \hspace{-0.5cm}
    \begin{subfigure}{0.43\textwidth} 
        \includegraphics[width=\linewidth]{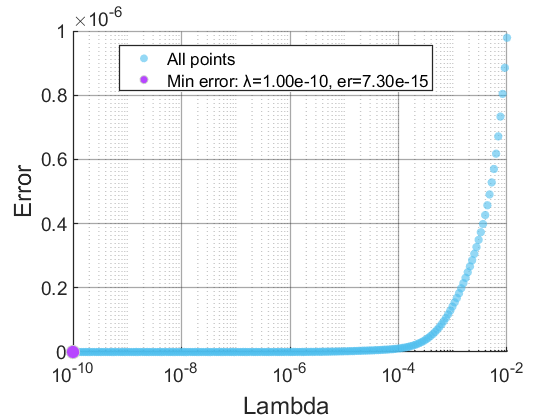}
    \end{subfigure}
    \caption{Residual energy decay (left) and the regularization parameter (right).}
    \label{couple2}
\end{figure}
\subsection{Continuous-time PAREs}
\label{sec4.4}For the parametric Riccati equations in our non-intrusive framework, we cannot provide the stabilizing initial guess that Newton's method typically requires to converge to the unique symmetric positive semidefinite solution. We therefore assume $A(\mu)$ is stable over $\mathcal{P}$, which justifies a zero initial guess and ensures convergence to the correct solution in reduced-order models.

The parameter domain is defined as $\mathcal{P}=[0.1,5]$. We consider the case of $s=1$, in which the nonlinear operator 
$\mathcal{L}_1$ is defined as 
$$
\mathcal{L}_1(X;\mu)=A^\top(\mu)X(\mu)+X(\mu)A(\mu)-X(\mu)G(\mu)X(\mu).
$$
In this setting, equation $\eqref{1}$ reduces to a continuous-time PARE,
where $G(\mu)=B(\mu)B^\top(\mu),$ 
$$
A(\mu)=\frac{1}{\mu}A_1+\frac{1}{\mu^2}A_2,\ B(\mu)=\mu B_1,\ M(\mu)=\frac{1}{\mu}M_1,
$$
$$
A_1=
\begin{bmatrix}
    -30 & -3\\
    2 & -30 & -3\\
    & \ddots & \ddots & \ddots\\
    & & 2 & -30 & -3\\
    & & & 2 & -30
\end{bmatrix},\
A_2=
\begin{bmatrix}
    0 & -3 & -4\\
    2 & 0 & -3 & -4\\
    2 & 2 & 0 & -3 & -4\\
    & \ddots & \ddots & \ddots & \ddots & \ddots\\
    & & 2 & 2 & 0 & -3 & -4\\
    & & & 2 & 2 & 0 & -3 \\
    & & & & 2 & 2 & 0
\end{bmatrix},
$$
$$
 B_1=
\begin{bmatrix}
    0.2 & \cdots & 0.2
\end{bmatrix}^T,\ 
M_1=
\begin{bmatrix}
    0.1 & \cdots & 0.1
\end{bmatrix}.
$$
\begin{figure}
    \centering
    \begin{subfigure}{0.43\textwidth} 
        \includegraphics[width=\linewidth]{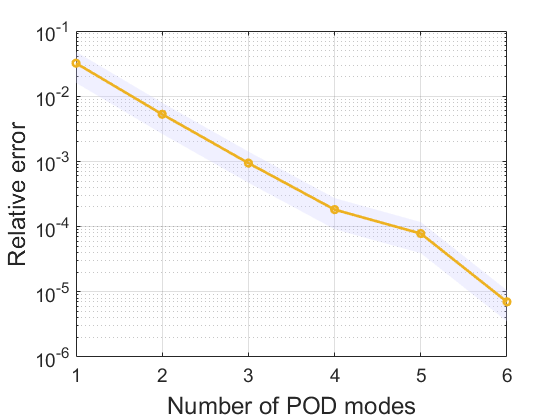}
        \caption{training parameters}
        \label{512_train}
    \end{subfigure}
    \hspace{-0.5cm}
    \begin{subfigure}{0.43\textwidth} 
        \includegraphics[width=\linewidth]{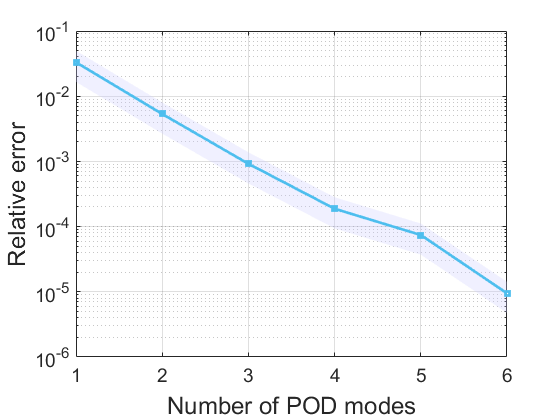}
        \caption{test parameters}
        \label{512_test}
    \end{subfigure}
    \caption{The average relative error, evaluated at the full-order system dimension of $N=512^2$ over $10^4$ uniformly distributed random test parameters.}
    \label{Riccati2}
\end{figure}

We construct the snapshot matrix and select the first 6 POD modes. Due to the nonlinear nature of the matrix equation, the quadratic term has a dimension of $n^2(n^2+1)/2$. Moreover, consolidating the coefficients of redundant terms in the quadratic operator involves complicated intermediate transformations. In this context, the advantages of the non-intrusive approach are clearly demonstrated. 

We generate $10^4$ uniformly distributed test parameters. As shown in \Cref{512_train} and \Cref{512_test}, OpInf maintains good accuracy with a low average relative error for $n=512$. We also provide average relative errors with different-sized POD modes and the CPU time required to compute the solutions using both OpInf and the MATLAB care solver. The results summarized in \Cref{tab:comparison3} show that OpInf is computationally more efficient than the conventional method. 
\Cref{10 test} presents a comparative analysis between solutions obtained using the OpInf and MATLAB care solver over 10 randomly chosen parameter instances. The close agreement between the two approaches validates the reliability and effectiveness of the method in practical settings.
\begin{table}
\centering
\caption{The average relative error and a comparison of CPU time between the MATLAB care solver (Reference method) and the OpInf, evaluated at the full-order system dimension of $N=512^2$ over $10^4$ uniformly distributed random test parameters.}
\label{tab:comparison3}

\begin{tabular}{cccc}
\toprule[1.5pt]
\multirow{2}{*}{\textbf{Strategies}} & 
\multicolumn{2}{c}{\textbf{The OpInf method}} & 
\multirow{2}{*}{\textbf{Reference method}} \\
\cmidrule{2-3} 
& \textbf{\(r = 4\)} & \textbf{\(r = 6\)} & \\ 
\midrule 

\(er_{avg}\) & \(1.88\times 10^{-4}\) & \(9.48\times 10^{-6}\)\\
\midrule
\(T_{\text{off}}\) & 512.68\,s & 512.68\,s & - \\
\midrule
\(T_{\text{on}}\) & 1.16\,s & 1.36\,s & - \\
\midrule
\(T_{\text{tot}}\) & 513.84\,s & 514.04\,s & \(1.03 \times 10^{5}\,s\) \\
\midrule
\(T_{\text{avg}}\)  & \(1.16\times 10^{-3}\,s\) & \(1.36\times 10^{-3}\,s\) & \(1.03 \times 10^{1}\,s\)\\
\bottomrule[1.5pt] 
\end{tabular}
\end{table}
\begin{figure}
     \centering    \includegraphics[width=0.65\linewidth]{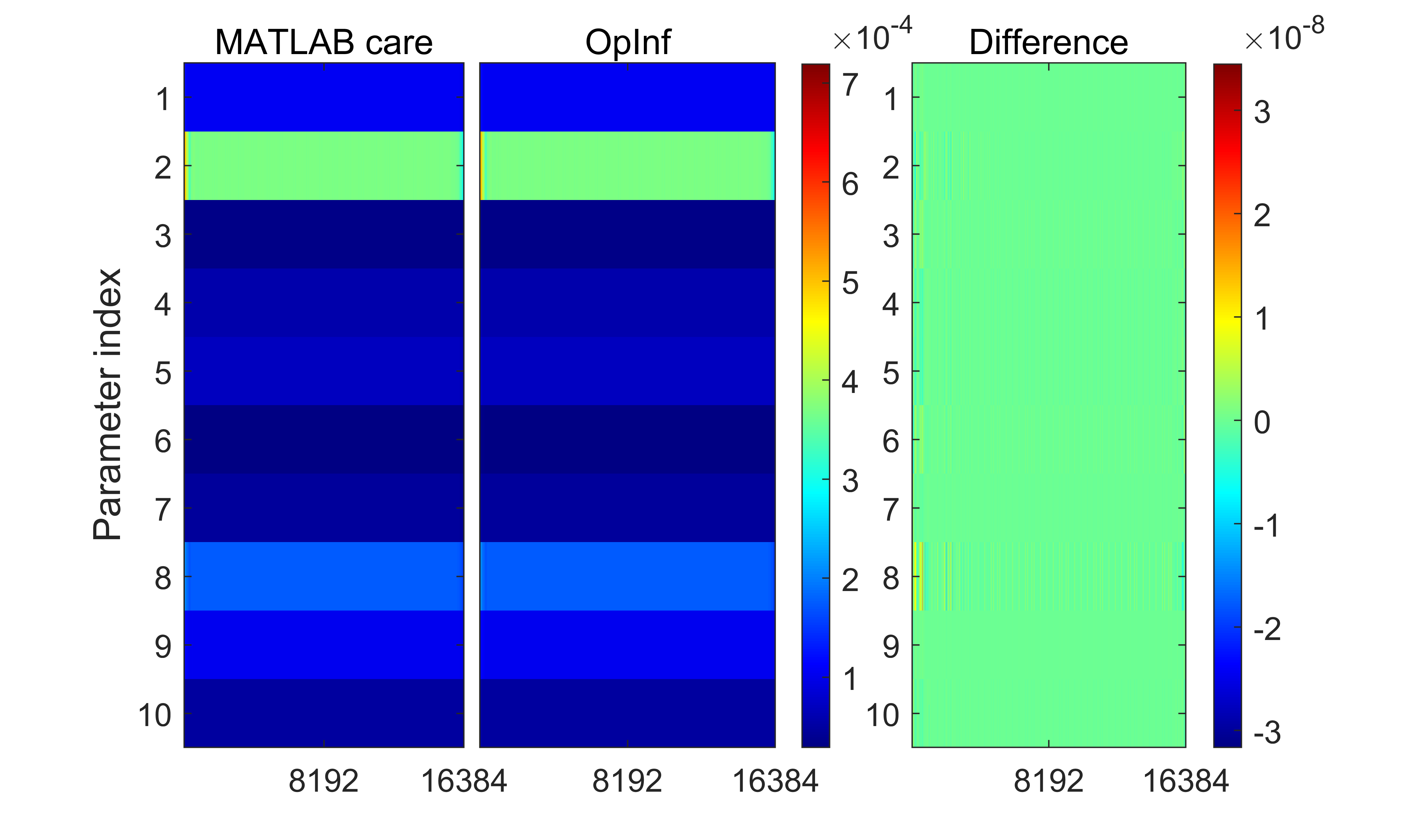}
     \caption{The solutions between the OpInf method and MATLAB care solver, evaluated at the full-order system dimension of $N=128^2$ over $10$ random test parameters.}
     \label{10 test}
\end{figure}
\section{Summary}
\label{sec5}
In this study, we employ a data-driven OpInf approach to establish a unified and efficient surrogate modeling framework for the rapid solution of parametric matrix equations. This framework is applicable to linear, nonlinear, and coupled systems of matrix equations. It requires neither access to the full-order constant operators nor knowledge of intermediate transformation procedures. Numerical experiments demonstrate that OpInf achieves high accuracy, as reflected in the average relative error, and significantly improves computational efficiency compared to conventional solvers, as evidenced by CPU time comparisons. 
However, the quality of the reduced-order model (accuracy and stability) remains highly dependent on the selection of training parameters, namely the solution snapshots. Our current work applies operator inference exclusively to the four classes of matrix equations presented herein. For more complex nonlinear matrix equations, we will explore alternative data-driven model reduction methodologies in future research.
\section*{Acknowledgments}
We would like to acknowledge the assistance of volunteers in putting together this example manuscript and supplement, as well as the computational resources provided by the High Performance Computing Platform of Xiangtan University.

\bibliographystyle{siamplain}
\bibliography{references}
\end{document}